\documentclass[11pt]{amsart}
\usepackage{amsmath,amsfonts,amssymb,amsthm,stmaryrd}
\usepackage{enumitem}
\usepackage{bbm}
\usepackage{tikz}
\usetikzlibrary{decorations.pathreplacing, positioning}
\usepackage{graphicx}
\usepackage{pgfplots}
\usetikzlibrary{backgrounds}
\usepackage{xcolor}
\usepackage{yfonts}
\usepackage{setspace}
\usepackage{comment}
\usepackage[T2A,T1]{fontenc}     
\newtheorem{theorem}{Theorem}
\newtheorem{lemma}[theorem]{Lemma}
\newtheorem{corollary}[theorem]{Corollary}
\newtheorem{proposition}[theorem]{Proposition}
\theoremstyle{definition}
\newtheorem{definition}[theorem]{Definition}
\newtheorem*{theorem*}{Theorem}
\newtheorem*{corollary*}{Corollary}
\newtheorem{remark}[theorem]{Remark}

\newcommand{\mbbn}{\mathbb N}

\newcommand{\mbbr}{\mathbb R}
\newcommand{\eps}{\varepsilon}
\def\diam{\operatorname{diam}}
\def\Card{\operatorname{Card}}
\def\mod{\operatorname{Mod}}
\def\dist{\operatorname{dist}}
\def\H{{\mathcal H}}

\usepackage{scalerel}[2014/03/10]
\usepackage[usestackEOL]{stackengine}
\def\intavg{\,\ThisStyle{\ensurestackMath{%
    \stackinset{c}{0\LMpt}{c}{0\LMpt}{\SavedStyle-}{\SavedStyle\phantom{\int}}}%
    \setbox0=\hbox{$\SavedStyle\int\,$}\kern-\wd0}\int}

\numberwithin{theorem}{section} \numberwithin{equation}{section}
\numberwithin{figure}{section}
\counterwithin{figure}{section}

\thanks{P.~K.~has been supported by the Research Council of Finland, grant number 364210.}
\author{Behnam Esmayli}
\address{Department of Mathematics, Link\"oping University, SE-581 83 Link\"oping, Sweden
}
\email{behnam.esmayli@liu.se}
\author{Pekka Koskela}
\address{Department of Mathematics and Statistics, P.O. Box 35 (MaD), FI-40014, University of Jyv\"askyl\"a, Finland
}
\email{pekka.j.koskela@jyu.fi}

\author{Khanh Nguyen}
\address{University of Science, Vietnam National University, Hanoi, Vietnam
}
\email{khanhnn@vnu.edu.vn}
\begin{document}
\title[Exceptional Sets for Quasiconformal Mappings]{Exceptional Sets for Quasiconformal Mappings in General Metric Spaces II}

\begin{abstract}
A homemorphism between domains in $\mathbb R^n$, $n\ge 2$ is quasiconformal, with its intricate analytic and geometric consequences, if the (pointwise) linear dilatation -- a purely metric quantity -- is uniformly bounded. Gehring proved that it will suffice to verify the uniform bound up to a set of measure zero as long as we can show that the dilatation is finite outside a subset of finite Hausdorff--$(n-1)$ measure. In short, we say that we can allow an exceptional codimension $1$ subset. 

In the metric setting, it has been proved, roughly speaking, that one can allow an exceptional codimension $p$ subset, $p \ge 1$, if the source space satisfies a $p$-Poincar\'e inequality.

We prove, effectively, the sharpness of the latter claim.
\end{abstract}

\maketitle

\tableofcontents
\section{Introduction}
Quasiconformal homeomorphisms between domains in $\mathbb R^n$, $n\ge 2$, enjoy many regularity properties. They are absolutely continuous on almost every line, belong to Sobolev spaces of higher order than in their definition, are (locally) H\"older continuous, have the Lusin N property, etc. These properties are tied to and follow from the equivalence of the so-called \textit{metric, analytic, and geometric} definitions of quasiconformality in $\mathbb R^n$. Fundamental works in the Euclidean setting are due to Ahlfors, Gehring, and V\"ais\"al\"a.

Study of quasiconformal mappings in non-Riemannian settings (e.g.\ in the Heisenberg groups and more general Carnot groups) was motivated by complex dynamics and rigidity questions in geometry. Important results were obtained by Mostow, Pansu, Kor\'anyi, Bourdon, Pajot, and many others. We refer to the introduction of~\cite{bal-kosk-rog} for references and a brief history of the major developments in the theory of quasiconformal mappings in non-Euclidean settings.

Heinonen and Koskela~\cite{Hei-Kosk-95} found techniques that were metric and did not depend on the special structures (e.g.\ foliation) of the underlying Euclidean or Carnot groups. In~\cite{HeiKo-Acta}, they introduced the so-called ``metric spaces with controlled geometry'' which unified many of the existing results about quasiconformal mappings.
\subsection{Removability of sets in metric quasiconformality} Let $f\colon X \to Y$ be a homeomorphism between two metric spaces (without any isolated points). For $r>0$, set
\begin{align*}
L_f(x,r) =&\sup\{d(f(y),f(x)): d(y,x)\le r\}, \\
l_f(x,r) =&\inf\{d(f(y),f(x)): d(y,x)\ge r\},    
\end{align*}
and let
\[
H_f(x)=  \limsup\limits_{r \to 0} \frac{L_f(x,r)}{l_f(x,r)}.
\]
We say that a homeomorphism $f\colon X \to Y$ is \emph{metrically quasiconformal} if there exists some $ 1 \leq H < \infty$ such that $H_f(x)\le H$ for every $x \in X$.

The definition, requiring a distortion control at \emph{every} point and \emph{all} small scales, is quite inconvenient to verify in practice. In the Euclidean setting, already Gehring~\cite{Geh1,Geh:63} proved that it suffices to verify the bound {a priori} only outside a set of $\sigma$-finite $\H^{n-1}$-measure.

It was proved in Heinonen-Koskela~\cite{Hei-Kosk-95} that, in the Euclidean setting, one can replace $\limsup$ with $\liminf$, i.e.\ for quasiconformality it suffices to have a uniform bound on
$$
h_f(x)=  \liminf\limits_{r \to 0}  \frac{L_f(x,r)}{l_f(x,r)},
$$
i.e., a bound that is a priori on \emph{some} sequence of small scales.

Finally, Kallunki-Koskela~\cite{Kallunki-Kosk-2000, Kallunki-Kosk-2003} proved that one can work with $h_f$ and at the same time allow an exceptional set:
\begin{theorem}[{\cite[Corollary 1.2]{Kallunki-Kosk-2003}}]
Let $f \colon \Omega \to \Omega'$ be a homeomorphism between domains in $\mathbb{R}^n, n\ge 2$. Suppose that $E \subset \Omega$ has $\sigma$-finite $\H^{n-1}$-measure, $h_f(x) <\infty$ at every $x \in \Omega \setminus E$, and for some $H < \infty$,  $h_f(x) \le H$ almost everywhere on $\Omega$. Then $f$ is quasiconformal, i.e.\  for some $H'<\infty$, we have $H_f(x) \le H'$  for every $x \in \Omega$. 
\end{theorem}
Such reductions in assumptions were motivated by and found applications in complex dynamics, see e.g.\ \cite{przytycki-rohde,haissinsky,smania,kozlovski-shen-strien,graczyk-smirnov,graczyk-smirnov}. Observe that the assumptions on $h_f$ are stronger than simply requiring $h_f(x) \le H$ almost everywhere. In fact, the homeomorphism $f$ defined in \eqref{eq:QC-map-f}, when viewed as on $\mathbb R^2$, shows the necessity of assuming more that $\operatorname*{ess\,sup} h_f < \infty$.

In metric setting, one is interested in minimal assumptions on $H_f$, or preferably on $h_f$, that still imply regularity properties for homeomorphisms $f\colon X \to Y$, for example, quasisymmetry, Sobolev regularity, etc. As in the Euclidean setting, study of such properties uses the relationships between \textit{metric}, \textit{analytic}, and \textit{geometric} (modulus) definitions of quasiconformality. ``A key point in connecting these definitions is absolute continuity on paths'' (\cite{kosk-wild-survey}).  

The starting point of our investigation is the following result due to Balogh-Koskela-Rogovin:
\begin{theorem}[\cite{bal-kosk-rog,Kosk-wild}]\label{bakoro}
    Let $X$ and $Y$ be proper metric measure spaces with $X$ locally Ahlfors $Q$-regular, $Q>1$, and let $f\colon X \to Y$ be a homeomorphism. Suppose that there exists a set $E \subset X$ with $\sigma$-finite $\H^{Q-p}$-measure, $1\le p<Q$, such that $h_f(x)<\infty$ at every $x \in X\setminus E$, and there exists $H<\infty$ such that $h_f(x) \le H$ a.e.\ on $X$. Finally, assume that $Y$ is \emph{locally Ahlfors $Q$-regular off $f(E)$}. Then, $f \in N^{1,p}_{loc}(X,Y)$. If, in addition, $X$ supports a $p$-Poincar\'e inequality, then $f \in N^{1,Q}_{loc}(X,Y)$ and, in particular, $f$ is absolutely continuous on $Q$-a.e.\ path.
\end{theorem}
Recall that proper means that closed bounded sets are compact. Other terminology will be explained in Section~\ref{sec:prelim}. The less standard notion of $Y$ being ``locally Ahlfors $Q$-regular off $f(E)$'' is given in Definition~\ref{def:Ahlfors-off}. In our settings, being in the Newtonian-Sobolev space $N^{1,p}_{loc}$ is simply equivalent to having an $L^p_{loc}$-integrable upper-gradient.

For more recent results of this nature, see~\cite{Williams:14}, \cite{Ntala24:metric-def} and~\cite{lahti:Zhou:24}. The homeomorphism in Theorem~\ref{bakoro}, under some additional natural assumptions on the spaces, enjoys other regularity properties. For example, if $X$ satisfies a  $p$-Poincar\'e inequality and $Y$ is so-called \textit{linearly locally connected}, then \cite[Theorem~5.2, Remark~5.3]{bal-kosk-rog} show that $f$ is analytically quasiconformal and it is locally quasisymmetric, hence geometrically quasiconformal. We will not pursue such details and focus only on Sobolev regularity in the rest of this paper.

\subsection{Main results}
It is worth remarking that the $N^{1,p}$-regularity part of the conclusion of Theorem~\ref{bakoro} does not require any Poincar\'e inequality. However, Koskela-Wildrick~\cite{Kosk-wild} showed that $N^{1,Q}$-regularity can fail in the absence of a $p$-Poincar\'e inequality. They prove this by constructing an example in the plane, which is the inspiration for the construction of our main object in the current paper.

Although their example cannot support a $p$-Poincar\'e inequality due to Theorem~\ref{bakoro}, Koskela-Wildrick~\cite{Kosk-wild} conjecture that it may support some weaker Poincar\'e inequality, i.e.\ a $q$-Poincar\'e inequality for some $q>p$. Using our construction, we confirm their conjecture, as the two constructions are practically bi-Lipschitz equivalent (see Section~\ref{sec:bilip-kosk-wil}). 

We prove not only the necessity of a $p$-Poincar\'e inequality on $X$, but also the sharpness of $p$ in order to have the full Sobolev regularity in Theorem~\ref{bakoro}.
\begin{theorem}\label{main2}
 For every $ p \in (1,2)$ and every $\eps >0$, there exist compact spaces $X$ and $Y$, a closed set $E \subset X$, and a homeomorphism $f\colon X \to Y$ such that
\begin{itemize}[topsep=0.3ex]
    \item $X$ is compact, geodesic, and Ahlfors $2$-regular,
    \item $Y$ is locally Ahlfors $2$-regular off $f(E)$,
    \item $0< \H^{2-p}(E) < \infty$,
    \item $H_f(x) = 1$ at every $x \in X \setminus E$,
    \item $X$ supports a $(p+\eps)$-Poincar\'e inequality,
\end{itemize}
but for every $q \ge p+\eps$, there exists a family of paths with positive $q$-modulus such that $f$ is not absolutely continuous on them. In particular, $ f \notin N^{1,q}_{loc}(X,Y)$. Moreover, $X$ and $Y$ can be chosen to be compact subsets of $\mathbb R^2$, with the induced metric and measure.
\end{theorem}
In order to prove Theorem~\ref{main2} we will construct a two-parameter family $\mathbf X:=\mathbf{X}(\lambda,\nu),\, \lambda \in (0,1/2),\, \nu \in \mathbb N$, of subsets of $\mathbb R^2$ that are compact and Ahlfors $2$-regular. Our second main result (Theorem~\ref{main1}) proves that $\mathbf X$ supports a $p$-Poincar\'e inequality if and only if $p > \textswab{p}_0$, where $\textswab{p}_0$ is explicitly given by formula \eqref{eq:p0-def}, and in particular depends only on $\lambda$ and $\nu$.

Given $ p \in (1,2)$, by choosing the parameters $\lambda$ and $\nu$ appropriately, we will have $ \textswab{p}_0 \in (p,p+\eps)$, which completes the requirements on $X$ in Theorem~\ref{main2}. The homeomorphism $f$ and $Y=f(X)$ will be as in~\cite{Kosk-wild}, and $E \subset X$ will be the Cartesian product of two Cantor sets (Section~\ref{sec:qc}).

The most involved part of proving Theorem~\ref{main2} is confirming the Poincar\'e inequality on $\mathbf X$, i.e., Theorem~\ref{main1}, which occupies most of the paper. Theorem~\ref{main1} and its proof can be of independent interest. We prove the Poincar\'e inequality on $\mathbf X$ by constructing explicit Semmes's pencils of curves.

\begin{remark}
    Analogues of our spaces $\mathbf{X}$ can be constructed in higher dimensions and used to prove higher dimensional versions of Theorem~\ref{main2}.
\end{remark}

\subsection{Comparison to non-self-similar Sierpi\'nski carpets}
Our methods are somewhat similar to those employed in \cite{MacTysWil}, but here we carry the calculations until we obtain well-known pointwise estimates, which are sufficient for the validity of the Poincar\'e inequality. In \cite{MacTysWil}, the authors appeal to a characterization of Poincar\'e inequality via modulus bounds, due to Keith~\cite{keith-mod}, and choose to use their pencils of curves to obtain these modulus bounds rather than the pointwise inequalities.

However, our spaces $\mathbf X$ and the Sierpi\'nski carpets in~\cite{MacTysWil} are constructed very differently. While the carpets arise from \emph{deleting} a countable collection of squares, our spaces $\mathbf X$ are constructed from \emph{adding} (union of) countably many squares. As a result, while one obtains carpets as a Gromov-Hausdorff limit of its finite approximations, a fact that is utilized in the proofs in~\cite{MacTysWil}, it is not obvious how one can approximate our $\mathbf X$ with a sequence of, say, ``nice'' Lipschitz domains in $\mathbb R^2$.

\subsection*{Notation} The inclusion $A \subset B$ allows for the possibility $A=B$. If we wish to exclude this, we write $A \subsetneq B$. The characteristic function of $E$ is denoted by $\chi_E$.

When $\mathcal{Q}$ is a collection of, say, cubes, we write $ \sum_{\mathcal{Q}}$ rather than $\sum_{Q \in \mathcal{Q}}$, and $\cup \mathcal{Q}$ (union) rather than  $\cup_{Q \in \mathcal{Q}} Q$.

We will denote the length of an interval $I \subset \mathbb R$ by $|I|$. By \emph{the canonical bijection} between two compact intervals we will mean the unique increasing affine bijection between them. The unit interval $[0,1]$ is denoted by $I_{01}$.

The Euclidean distance of $p$ and $q$ in $\mathbb R^n, n \ge 2$, is denoted by $\|p-q\|$. We write $\diam E$ for the diameter of a set $E$ in a metric space. We denote by $\dist(F,E)$ the distance of sets. When $F=\{x\}$, we simply write $\dist(x,E)$.

A ball in $(X,d)$ comes with a pre-determined center $x$ and radius $r>0$. Given the ball $B(x,r)$ and constant $\kappa >0$, we write $\kappa B:=B(x,\kappa r)$.

We use the words ``path'' and ``curve'' interchangeably. Paths are always defined on compact intervals. A path $\gamma$ is rectifiable if it has finite length, denoted by $\ell(\gamma)$.

By $\H^\alpha(E)$ we mean the $\alpha$-dimensional Hausdorff measure of $E$, and $\dim E$ stands for the Hausdorff dimension of $E$. The intrinsic metric on our space will be bi-Lipschitz equivalent to the inherited Euclidean metric, so, the Hausdorff dimension is the same with respect to either one.

When $0 < \mu(E) < \infty$, for measurable $u$ we will use the shorthand
$$
\intavg_E u\, d\mu :=\frac{1}{\mu(E)} \int_E u\, d\mu.
$$
We also use $u_E$ to denote this quantity.

By $C$ we will denote a generic constant whose value may change even in a
 single string of estimates. For non-negative quantities $A=A(a,b,\ldots)$ and $B=B(a,b,\ldots)$, we write $A \lesssim B$ to mean that there exists $C>0$, called ``the comparison constant'', that is independent of $a,b,\ldots$ and $A \le C B$ for all choices of these parameters. It should be clear from the context what the parameters $a,b,\ldots$ are. For example, if one is discussing Ahlfors regularity of a measure $\mu$, then by $\mu(B(x,r)) \lesssim r^Q$ it is understood that the comparison constant is independent of $x$ and $r$. We write $A \approx B$ if $A \lesssim B$ and $B \lesssim A$. 
\section{Preliminaries}\label{sec:prelim}
Throughout, $(X,d,\mu)$ will be a metric space, with no isolated points, equipped with a Borel measure $\mu$.
\subsection{Ahlfors regularity}
We say $\mu$ is \emph{doubling} if there exists a constant $C>0$ such that
$$
\mu(2B) \le C \mu(B), \quad \text{for every ball $B \subset X$}.
$$
We further require that for one ball, hence for all balls, $0 < \mu(B) < \infty$. We call $C$ \textit{the doubling constant} of $\mu$.

A stronger condition is Ahlfors regularity. We say that $\mu$ is \emph{Ahlfors $Q$-regular} if for every $x \in X$ and every $0<r<\diam X$ we have
$$
C^{-1}r^Q \le \mu(B(x,r)) \le Cr^Q,
$$
for a constant $C>0$ independent of $x$ and $r$.  We say that $X$ is \emph{locally Ahlfors $Q$-regular} if for every compact $A \subset X$ there exist constants $C>0$ and $ r_0 > 0 $ such that
    $$
    \forall\, x \in A, \; \forall\, r \in (0, r_0) : \; C^{-1}r^Q \le \mu(B(x,r)) \le Cr^Q.
    $$
\begin{definition}[\cite{Kosk-wild}]\label{def:Ahlfors-off}
   Let $ F \subset X $ be non-empty. We say that $X$ is \emph{Ahlfors $Q$-regular off $F$} if there exists a constant $C>0$ such that for every $x \in X \setminus F$ there exists $r_x>0$ such that
    $$
    \forall\, r \in (0, r_x) \, : \quad  C^{-1}r^Q \le \mu(B(x,r)) \le Cr^Q.
    $$
\end{definition}

\subsection{Absolute continuity and upper-gradients}
For a rectifiable path in $X$ we denote by $\gamma_s\colon [0,\ell(\gamma)] \to X$ the arc-length parameterization of $\gamma$. Given a Borel function $g \colon X \to [0,\infty]$ and a rectifiable path $\gamma$, we define 
$$
\int_\gamma g\, ds := \int_0^{\ell(\gamma)} g(\gamma_s(t))\, dt.
$$
Let $(Y,d)$ be another metric space. We say that a function $u\colon X \to Y$ is \emph{absolutely continuous} on a rectifiable curve $\gamma$ in $X$ if the composition $u \circ \gamma_s \colon [0,\ell(\gamma)] \to Y$ is absolutely continuous.

 We say that a Borel function $g\colon X \to [0,\infty]$ is an upper-gradient of a continuous mapping $u \colon X \to Y$ if
\begin{equation}\label{defeq:upper-grad}
    d(u(y),u(x)) \le \int_\gamma g\, ds,
\end{equation}
for every $x,y$ in $X$ and every rectifiable curve $\gamma$ that joins $x$ and $y$. We say that $g$ is an upper-gradient for a continuous mapping $u$ defined on (or restricted to) $U\subset X$ if \eqref{defeq:upper-grad} holds for all rectifiable curves in $U$.

The following is easy to prove from the definitions and the so-called absolute continuity of integral.
\begin{lemma}\label{lem:may29-1}
    Suppose $g$ is an upper-gradient of a continuous mapping $u\colon X \to Y$ and $\int_\gamma g\, ds < \infty$ for a rectifiable curve $\gamma$. Then $u$ is absolutely continuous  on $\gamma$. 
\end{lemma}
Let $\Gamma$ be a collection of paths in $X$. For $1 \leq p < \infty$, \textit{the $p$-modulus} of $\Gamma$ is defined as
$$
\mod_p (\Gamma) := \inf_g \int_X g^p\, d\mu ,
$$
where the infimum is taken over all Borel functions $g\colon X \to [0,\infty]$ such that $\int_\gamma g\, ds \ge 1$ for all rectifiable curves $\gamma \in \Gamma$.

We say that a certain property holds for \emph{$p$-almost every curve} in $X$, written \emph{$p$-a.e.\ curve}, if the collection of curves for which it fails has zero $p$-modulus. By H\"older's inequality, if $\mu(X)<\infty$ then a property that holds for $p$-a.e.\ curve also holds for $q$-a.e.\ curve if $1 \leq q<p$; but, of course, the converse does not necessarily hold.
\begin{lemma}[{\cite[Lemma 5.2.8]{HKST:15}}]
\label{lem:fuglede}
    For every Borel function $ g \in L^p(X) $,  $\int_\gamma g\, ds < \infty$ for $p$-a.e.\ curve $\gamma$ in $X$.
\end{lemma}
From Lemmas~\ref{lem:may29-1} and~\ref{lem:fuglede} we obtain:
\begin{corollary}\label{cor:fuglede}
    If $g$ is an upper-gradient of a continuous mapping $u\colon X \to Y$ and $ g \in L^p(X) $, then $u$ is absolutely continuous on $p$-almost every curve.
\end{corollary}
When $X$ is locally compact and $Y$ is equipped with a locally finite Borel measure, as will be the case in our main results, we say that a continuous $u\colon X \to Y$ is in $N^{1,p}_{loc}(X;Y)$ if it has an upper-gradient that belongs to $L^p_{loc}(X)$. We will not need further properties of the space $N^{1,p}_{loc}(X;Y)$, known as the Newtonian-Sobolev spaces in literature (\cite{Shanmu:00,HKST:15}).

By Corollary~\ref{cor:fuglede}, if $u \in N^{1,p}_{loc}(X;Y)$, then $u$ is absolutely continuous on $p$-a.e.~curve in $X$. Thus, if $u$ is not absolutely continuous on $p$-a.e.~curve in $X$ then $u \notin N^{1,p}_{loc}(X;Y)$.

\subsection{Poincar\'e inequality}
The question of absolute continuity on a.e.~curve in $X$ will not make much sense if the space has too few nontrivial rectifiable curves, or none at all, for example the Koch snowflake.

Poincar\'e inequality is a condition that guarantees that a space will have a significant supply of rectifiable curves between any two regions of it. There are many variant definitions of the Poincar\'e inequality in literature that mostly end up being equivalent. We follow the one in our main references \cite{bal-kosk-rog} and \cite{Kosk-wild}, which is quite standard.

We say that a metric measure space $(X,d,\mu)$ satisfies a $p$-Poincar\'e inequality, $1\le p <\infty$, if there exist constants $\kappa \ge 1$ and $C > 0$ such that for every ball $B(x,r)$, every bounded continuous function $u\colon \kappa B \to \mathbb R$ and every upper-gradient $g$ of $u$ on $\kappa B$, we have
\begin{equation}\label{defeq:p-PIi}
\intavg_B |u-u_B|\, d\mu \le C (\diam B) \Bigl( \intavg_{\kappa B} g^p \, d\mu \Bigr)^{1/p}.
\end{equation}
As mentioned already, Poincar\'e inequality is more a connectivity condition than an analytic one. This has manifested itself in multiple ways throughout the rich literature that is now available.

Indeed, the strongest connection is shown in \cite{keith-mod} where it is proved that the $p$-Poincar\'e inequality is equivalent to a quantitative lower bound on the $p$-modulus of path families connecting pairs of points in the space; up to some technical modifications.

\subsection{Pointwise inequalities}
One way to prove a Poincar\'e inequality is via the so-called pointwise inequalities. The following is proved by integrating the inequality \eqref{eq:pointwise-h} in both $y$ and $x$ variables.
\begin{lemma}
    Suppose $B \subset X$ is a ball, $u\colon B \to \mathbb R$ is continuous and for a measurable function $h \ge 0$ we have
    \begin{equation}\label{eq:pointwise-h}
        |u(x)-u(y)| \le d(x,y)(h(x)+h(y)), \quad \text{for a.e. $x,y\in B$}.
    \end{equation}
    Then
    \begin{equation*}\label{eq:PI-h}
    \intavg_B|u-u_B|\, d\mu \le 2(\diam B) \intavg_B h\, d\mu.
    \end{equation*}
\end{lemma}
The lemma suggests that pointwise inequalities translate to Poincar\'e-type integral inequalities. Thus, the idea that we follow in this paper is to obtain the pointwise inequalities \eqref{eq:pointwise-h} with $h$ taken as the truncated maximal function of $g^p$ where $g$ is an upper-gradient of $u$. The technical task of the current paper is to achieve such pointwise inequalities from carefully constructed pencils of curves. These methods are now quite standard in literature.

For a fixed $R>0$, the truncated maximal function is
$$
\mathcal{M}_R f(x):= \sup_{0<r\le R} \intavg_{B(x,r)} |f|\, d\mu.
$$
As an operator, $\mathcal M_R$ has very important boundedness properties (when the measure is doubling). But we will not need the specific details here and refer to \cite{Haj:Ko:met} for a summary that is most relevant to analysis on metric spaces.

The precise result that we will need is the following.
\begin{lemma}[\cite{Haj:Ko:met}, Theorem~3.3]\label{lem:pointwise-then-PI}
    Fix $p>1$. Let $\mu$ be a doubling measure on $X$ and let $B$ be a ball and $\kappa \ge 1$ be a constant. Suppose that $u$ is an integrable function on $(\kappa+1)B$ and $g \ge 0$ is an $L^p$-integrable function on $(\kappa+1)B$ with the property that there exists a set $N$ with $\mu(N)=0$ and a constant $C>0$ such that for all $x$ and $y$ in $(\kappa+1)B \setminus N$,
    $$
     |u(x)-u(y)| \le Cd(x,y)\Bigl((\mathcal M_{\kappa d(x,y)}g^p(x))^{1/p}+(\mathcal M_{\kappa d(x,y)}g^p(y))^{1/p}\Bigr).
    $$
    Then there exists a constant $C'>0$, depending only on $C$, $p$ and the doubling constant of $\mu$, such that
    $$
    \intavg_B|u-u_B|\, d\mu \le C' (\diam B) \Bigl(\intavg_{(\kappa+1)B} g^p \, d\mu\Bigr)^{1/p}.
    $$
\end{lemma}

\subsection{Coarea integration}\label{sec:coare-trapez}
In this subsection, $E$ is the trapezoidal domain in the plane between two parallel horizontal line segments $A$ and $B$ of lengths $\ell \le L$ such that their projections onto the $x$-axis intersect. Without loss of generality, let the larger segment be $A=[0,L]$ (Figure~\ref{fub:coarea}). For each $t \in [0,L]$, let $\gamma_t$ be the path that joins the point $(0,t)$ along a line to the point on the other segment that corresponds to it under the canonical bijection between the two segments.

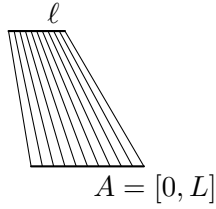
\begin{figure}[h]
    \centering
    \begin{tikzpicture}[line cap=round,line join=round,scale=0.3]
  \foreach \x in {0,0.5,1,1.5,2,2.5,3,3.5,4,4.5,5}
  {
    \draw (0.5*\x-1,6) -- (\x,0);
}
\draw[-, thick] (0,0) -- (5,0);
\draw[-, thick] (-1,6) -- (1.5,6);

\node at (5.5,-1) {$A=[0,L]$};
\node at (1,6.8) {$\ell$};
\end{tikzpicture}
    \caption{Coarea integration.}
    \label{fub:coarea}
\end{figure}

\begin{lemma}\label{lem:coarea-intgr}
For every measurable $g\colon E \to [0,\infty]$,
\begin{equation}
\label{fubp02}
\int_{[0,L]}\int_{\gamma_{t}} g \, dsdt \lesssim \frac{L}{\ell} \int_{E} g \, d\mathcal{L}^2,
\end{equation}
with a comparison constant that does not depend on $L$ and $\ell$.
\end{lemma}
\begin{proof}
    Let $\pi \colon E \to [0,L]$ be the function that satisfies $\pi(x,y)=t$ iff $\gamma_t$ is the path that contains $(x,y)$. Then $\pi$ is Lipschitz and by easy calculations $\|D\pi\| \lesssim L/\ell$, where $\|D\pi\|$ is the supremum norm of the derivative. Thus, by the coarea formula~(\cite{Evans-Gariepy})
    $$
    \int_{[0,L]}\int_{\gamma_{t}} g \, dsdt = \int_{E} g \|D\pi\| \, d\mathcal{L}^2 \lesssim \frac{L}{\ell} \int_{E} g \, d\mathcal{L}^2. 
    $$
\end{proof}

\begin{corollary}\label{cor:coare-max}
    In the above setting, further assume that for some constant $C\ge 1$ and a point $x \in \mathbb R^2$ we have $\frac{L}{\ell} \le C$, and $B(x,CL) \supset E$. Then with constants that depend only on $C$, we have, for every measurable $g\colon B(x,CL) \to [0,\infty]$,
\begin{equation*}
    \intavg_{[0,L]}\int_{\gamma_{t}} g \, dsdt \lesssim L \mathcal M_{CL}g(x).
\end{equation*}
\end{corollary}
\begin{proof}
    From \eqref{fubp02} and the assumptions, we have
    \begin{align*}
    \intavg_{[0,L]}\int_{\gamma_{t}} g \, dsdt &\lesssim \frac{L^2}{\ell} \frac{1}{L^2} \int_{E} g \, d\mathcal{L}^2 \\
    & \lesssim L\frac{L}{\ell} \frac{1}{L^2} \int_{B(x,CL)}g\, d\mathcal{L}^2\\
    & \lesssim L \mathcal M_{CL}g(x).
    \end{align*}
\end{proof}
An iteration of Corollary~\ref{cor:coare-max} yields an estimate even when $\frac{L}{\ell}$ is arbitrarily large. When $x$ is ``the vertex'' of the trapezoid, we can even let $\ell \to 0$ and obtain a quantitative approximation of integration in polar coordinates.
\begin{corollary}\label{cor:coare-max-2}
    Suppose that the trapezoid $E$ satisfies $\dist(B,A) \approx |A|=L$. Let $x \in \mathbb R^2$ be an arbitrary point such that $\dist(x,B) \approx \ell$ (Figure~\ref{fubfub:coarea}). Then with constants that depend only on the previous comparison constants, we have, for every measurable $g\colon B(x,CL) \to [0,\infty]$,
\begin{equation*}
    \intavg_{[0,L]}\int_{\gamma_{t}} g \, dsdt \lesssim L \mathcal M_{CL}g(x).
\end{equation*}
\end{corollary}

\begin{figure}[h]
    \centering
    \begin{tikzpicture}[line cap=round,line join=round,scale=0.3]
  \foreach \x in {0,0.5,1,1.5,2,2.5,3,3.5,4,4.5,5}
  {
  \draw (0.2*\x-2,15) -- (\x,0);
  }
\draw[-, thick] (0,0) -- (5,0);
\draw[-] (-1.2,9) -- (1.4,9);
\draw[-] (-1.6,12) -- (.2,12);
\draw[-] (-1.8,13.5) -- (-.4,13.5);
\draw[-] (-2,15) -- (0.2*5-2,15);

\node at (-2,16) {$|B|=\ell$};

\node at (4.5,-1) {$|A|=L$};
\node at (-2.4,6) {$\frac{L}{2}\approx$};
\node at (2.7,9) {$\approx \frac{L}{2}$};
\node at (1.7,12) {$\approx \frac{L}{4}$};
\node at (-2.7,10) {$\frac{L}{4}\approx$};
\node at (-2.6,13) {$\vdots$};
\end{tikzpicture}
    \caption{Truncated cone and the maximal function.}
    \label{fubfub:coarea}
\end{figure}
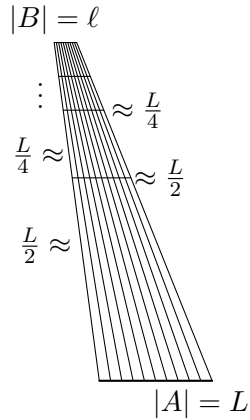

\subsection{Pencil of curves}\label{sec:pencils}
To illustrate and motivate the philosophy behind the proof of the Poincar\'e inequality in this paper, it is worthwhile to recall the basic idea behind using pencils of curves.

\subsubsection{Poinar\'e inequality on $\mathbb R^2$} Every $\mathbb R^n$ supports a $1$-Poinar\'e inequality where $n\in\mathbb N$. To prove this for $\mathbb R^2$, via pointwise inequalities, fix $x$ and $y \neq x$ in the plane. Without loss in generality, suppose $x=(0,\ell)$ and $y=(0,-\ell)$. Let $I=[-\ell,\ell] \times \{0\}$. For each $t \in I$ let $\gamma_t$ be the path that joins $x$ to $(t,0)$ and then to $y$ by straight lines (Figure~\ref{plane}).

\begin{figure}[h]
    \centering
\begin{tikzpicture}[scale=0.6]
\begin{pgfonlayer}{background}
  \fill[lightgray!30] (-3,-4) rectangle (3,4);
\end{pgfonlayer}

\def\n{10}          
\def\h{3}           
\def\l{4}           
\coordinate (top1) at (0,\h);
\coordinate (bottom1) at (0,-\h);

\foreach \i in {0,...,\n} {
  \pgfmathsetmacro\x{-\l/2 + \i*\l/\n}
  \draw[gray] (top1) -- (\x,0);
  \draw[gray] (bottom1) -- (\x,0);
}
\node at (0.5,3) {$x$};

\node at (0.5,-3) {$y$};

\node at (0.5,0.3) {$I$};
\node at (-1.1,0.6) {$\gamma_t$};
\node at (-1.2,-0.3) {$t$};
\draw[thick] (-1.2,0) circle (1pt);

\draw (-2,0)--(2,0);
\end{tikzpicture}
\caption{The pencil of curves from $x$ to $y$ in proving the $1$-Poincar\'e inequality on $\mathbb R^2$.}
\label{plane}
\end{figure}
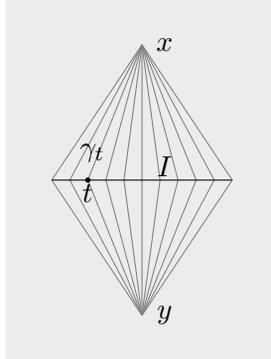

Now if $u$ is continuous on $\mathbb R^2$ and $g$ is an upper-gradient of $u$, then
$$
    |u(x)-u(y)| \leq \int_{\gamma_t} g\, ds, \quad \text{ for every $t \in I$}.
$$
Thus,
\begin{equation}\label{eq:a-monday3}
    |u(x)-u(y)| \leq \intavg_I \int_{\gamma_t} g\, dsdt.
\end{equation}
It is by now a standard fact (e.g.\ see \cite[Exercise~4.9]{Hei:01}) that, with universal constants,
\begin{equation}\label{eq:a-monday4}
\intavg_I \int_{\gamma_t} g \chi_{\{y>0\}}\, dsdt \lesssim |I| \mathcal M_{C|I|} g(x).
\end{equation}
Indeed, this follows from Corollary~\ref{cor:coare-max-2}.

By symmetry, from \eqref{eq:a-monday3} and \eqref{eq:a-monday4} we obtain
\begin{equation*}
    |u(x)-u(y)| \leq \intavg_I \int_{\gamma_t} g\, dsdt \lesssim d(x,y) \Bigl(\mathcal M_{Cd(x,y)} g(x) + \mathcal M_{Cd(x,y)} g(y)\Bigr).
\end{equation*}
By Lemma~\ref{lem:pointwise-then-PI}, this implies a $1$-Poincar\'e inequality on $\mathbb R^2$. Similar argument works in the higher dimensional $\mathbb R^n$.

\subsubsection{Poincar\'e inequality on the bow-tie}\label{subsec:glue-planes}
Now consider the case of two quadrants meeting at a single point. We again use a pencil of curves to illustrate how one can prove that the resulting space supports a $p$-Poincar\'e inequality for (in fact, only for) $p>2$. We only sketch the main idea.

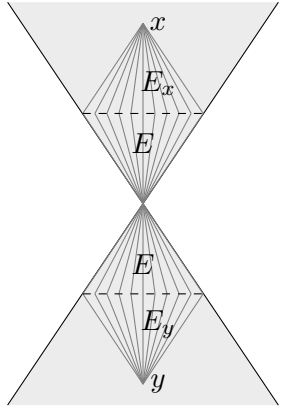
\begin{figure}[h]
    \centering
\begin{tikzpicture}[scale=0.4]
\draw (4.5,3.7)--(0,-3);
\draw (-4.5,3.7)--(0,-3);
\begin{pgfonlayer}{background}
  \fill[lightgray!30]
  (-4.5,3.7) -- (4.5,3.7) -- (0,-3) -- cycle;
\end{pgfonlayer}
\draw (4.5,-9.7)--(0,-3);
\draw (-4.5,-9.7)--(0,-3);
\begin{pgfonlayer}{background}
  \fill[lightgray!30]
    (-4.5,-9.7) -- (4.5,-9.7) -- (0,-3) -- cycle;
\end{pgfonlayer}

\def\n{10}          
\def\h{3}           
\def\l{4}           

\coordinate (top1) at (0,\h);
\coordinate (bottom1) at (0,-\h);

\foreach \i in {0,...,\n} {
  \pgfmathsetmacro\x{-\l/2 + \i*\l/\n}
  \draw[gray] (top1) -- (\x,0);
  \draw[gray] (bottom1) -- (\x,0);
}
\coordinate (top2) at (0,-\h);           
\coordinate (bottom2) at (0,-3*\h);      

\foreach \i in {0,...,\n} {
  \pgfmathsetmacro\x{-\l/2 + \i*\l/\n}
  \draw[gray] (top2) -- (\x,-2*\h);
  \draw[gray] (bottom2) -- (\x,-2*\h);
}
\node at (0.5,3) {$x$};
\node at (0.5,1) {$E_x$};

\node at (0,-1) {$E$};
\node at (0,-5) {$E$};

\node at (0.5,-9) {$y$};
\node at (0.5,-7) {$E_y$};

\draw[dashed] (-2,0)--(2,0);
\draw[dashed] (-2,-6)--(2,-6);
\end{tikzpicture}
\caption{The pencil of curves from $x$ to $y$ is forced to ``pinch'' and lose volume, resulting in a weaker Poincar\'e inequality.}
\label{fig:four-cone}
\end{figure}

So, let $X=\{(a,b): |b| \ge |a|\}$, equipped with the Euclidean metric and measure induced from $\mathbb R^2$. Fix $x=(0,\ell)$ and $y=(0,-\ell)$. Now consider a copy of the double cone of the previous example between $x$ and the origin $(0,0)$, and separately, a double cone between $y$ and the origin (Figure~\ref{fig:four-cone}). Because we can simultaneously index both path families over $I:=[-\ell/2,\ell/2]$, we can concatenate paths to form a family of paths, denoted $\Gamma=\{\gamma_t: t \in I \}$, from $x$ to $y$.

As before, if $u$ is continuous on $X$ and $g$ is an upper-gradient of $u$, then
\begin{equation}\label{eq:a-monday6}
    |u(x)-u(y)| \leq \intavg_{I} \int_{\gamma_t} g\, dsdt.
\end{equation}
Let $E_x$ be the cone at $x$ and $E_y$ the cone at $y$. Then by \eqref{eq:a-monday4},
\begin{equation}\label{eq:eq;1}
    \intavg_I \int_{\gamma_t} g \chi_{E_x}\, ds dt \lesssim |I| \mathcal M_{C|I|} g(x).
\end{equation}
Similarly,
\begin{equation}\label{eq:eq;2}
    \intavg_I \int_{\gamma_t} g \chi_{E_y}\, ds dt \lesssim |I| \mathcal M_{C|I|} g(y).
\end{equation}

Denote the union of the middle two cones that touch the origin by $E$. The best estimate that we can have here is via integration by polar coordinates:
\begin{equation}\label{eq:a-monday50}
    \intavg_{I} \int_{\gamma_t} g \chi_{E}\, ds dt \lesssim \int_{E} g \frac{1}{|y|} \, d\mathcal{L}^2.
\end{equation}

But notice that ``the coarea density'' ${1}/{|y|}$ blows up at the origin, which is away from our points $x$ and $y$. The way to bypass this problem is to apply H\"older's inequality. For $p>1$,
\begin{equation*}
   \int_{E} g \frac{1}{|y|} \, d\mathcal{L}^2 \lesssim \Bigl(\int_{E} |y|^{-p/(p-1)} \, d\mathcal{L}^2\Bigr)^{(p-1)/p} \Bigl(\int_{E} g^p \, d\mathcal{L}^2\Bigr)^{1/p}.
\end{equation*}
We have isolated $g$ from the coarea density. But $ \int_{E} |y|^{-p/(p-1)} \, d\mathcal{L}^2$ is finite iff $p>2$. When $p>2$,
$$
\Bigl(\int_{E} |y|^{-p/(p-1)} \, d\mathcal{L}^2\Bigr)^{(p-1)/p} \approx |I|^{1-\frac{2}{p}}.
$$
So, from \eqref{eq:a-monday50} and $\mathcal{L}^2(E)\approx |I|^2$ we get
\begin{align*}
 \intavg_{I} \int_{\gamma_t} g \chi_{E}\, ds dt & \lesssim |I| \Bigl(\intavg_{E} g^p \, d\mathcal{L}^2\Bigr)^{1/p} \\
 & \lesssim d(x,y) \Bigl((\mathcal M_{Cd(x,y)} g^p(x))^{1/p}+(\mathcal M_{Cd(x,y)} g^p(y))^{1/p}\Bigr).    
\end{align*}
By applying H\"older's inequality to the right-hand sides of \eqref{eq:eq;1} and \eqref{eq:eq;2} we have
\begin{equation*}
    \intavg_I \int_{\gamma_t} g \chi_{E_x}\, ds dt \lesssim |I| \mathcal M_{C|I|} g(x) \lesssim |I| (\mathcal M_{Cd(x,y)} g^p(x))^{1/p},
\end{equation*}
and analogously
\begin{equation*}
    \intavg_I \int_{\gamma_t} g \chi_{E_y}\, ds dt \lesssim |I| \mathcal M_{C|I|} g(y) \lesssim |I| (\mathcal M_{Cd(x,y)} g^p(y))^{1/p}.
\end{equation*}

Continuing from \eqref{eq:a-monday6}, we apply the last three inequalities to obtain the desired pointwise inequality:
\begin{align*}
 |u(x)-u(y)| & = \intavg_I \int_{\gamma_t} g \chi_{E_x}\, ds dt + \intavg_I \int_{\gamma_t} g \chi_{E}\, ds dt + \intavg_I \int_{\gamma_t} g \chi_{E_y}\, ds dt \\
            & \lesssim d(x,y) \Bigl((\mathcal M_{Cd(x,y)} g^p(x))^{1/p}+(\mathcal M_{Cd(x,y)} g^p(y))^{1/p}\Bigr).    
\end{align*}
By Lemma~\ref{lem:pointwise-then-PI}, this implies that $X$ supports a $p$-Poincar\'e inequality for $p>2$.

\subsection{The outline of our proof of the $p$-Poincar\'e inequality}
The key phenomena behind the example of the bow-tie space in the previous section, which will inspire the later constructions and analyses in our work, are as follows: 
\begin{enumerate}
    \item Due to the pinching in the space itself, the traces of our path families must lose measure at certain regions, which introduces unbounded coarea densities when comparing the double integrals to the area integrals;\footnote{The pinching at the endpoints $x$ and $y$ is not a problem due to \eqref{eq:a-monday4}.}
    \item We use H\"older's inequality to separate these coarea densities from the $L^p$-integrals of upper-gradients;
    \item The coarea densities are in $L^{p/(p-1)}$ only for certain range of $p>\textswab{p}_0$;
    \item If calculations work out, we obtain pointwise inequalities involving the truncated maximal function of $g^p$, hence, a $p$-Poincar\'e inequality follows by Lemma~\ref{lem:pointwise-then-PI}.
\end{enumerate}
Our proof of the Poincar\'e inequality for our space $\mathbf X$ will be a more involved execution of the above. In Section~\ref{sec:curve-fam} we construct pencils of curves between arbitrary regions of our space. In Section~\ref{sec:polar-int} we estimate the coarea densities associated with these path families. Using these, in Section~\ref{secpiyes} we prove the pointwise inequalities that are sufficient for the Poincar\'e inequality.

In our space $\mathbf{X}$ there are two competing factors. The first factor is that our space is made up of ``good parts'' that are glued to each other along not just single points but along copies of a Cantor set $\textswab{C}$. So, the pinching is not as extreme as in the bow-tie space. This helps to get a better Poincar\'e inequality, i.e.~ with $p<2$.

Indeed, it is well-known that if we glue two unit squares along a Cantor set $\textswab{C}$, then the space supports a $p$-Poincar\'e inequality for all $p>2-\dim \textswab{C}$. This is proved in  \cite[Section~6.14]{HeiKo-Acta}, but also follows easily from our construction of the path family $\Gamma(I)$ later, for which we calculate the coarea density and prove that it is integrable exactly when $p>2-\dim \textswab{C}$.

However, the second factor works against the Poincar\'e inequality: our paths have to cross the Cantor set $\textswab{C}$ not once or twice, but ``un-countably many times'' determined by a second Cantor set $\textswab{D}$. This intensifies the blowups of the coarea density for the path families, preventing a $p$-Poincar\'e inequality for $p$ too close to $1$.

Therefore, we should expect better Poincar\'e inequalities if 
$\textswab{C}$ has a larger Hausdorff dimension or $\textswab{D}$ has a smaller Hausdorff dimension
(Remark~\ref{thresh}).

The main technical achievement of this paper is that in the case of self-similar Cantor sets, we explicitly calculate the sharp $\textswab{p}_0$ such that our space satisfies a $p$-Poincar\'e inequality for all $p > \textswab{p}_0$, and for no $p \le \textswab{p}_0$.

\section{The Space $({\mathbf{X}},d,\mu)$}\label{construction}
In this section, we construct the main object of interest in this paper: a metric measure space $({\mathbf{X}},d,\mu)$ that corresponds to the choices of two Cantor sets $\textswab{C}$ and $\textswab{D}$ in $\mathbb R^1$.
\subsection{Notation regarding Cantor sets}\label{seccantor}
Given a constant $\lambda \in (0,1/2)$, we denote by $\textswab{C}(\lambda)$ the standard self-similar Cantor set constructed as follows. Step $k=0$ consists of the interval $I_{0,1}:=I_{01}=[0,1]$. For step $k=1$, we remove from $I_{0,1}$ the central open interval of length $1-2\lambda$. This leaves us with the two \emph{surviving} intervals $I_{1,1}$ and $I_{1,2}$, each of length $\lambda$. For step $k=2$ of the construction, we remove from each of the latter intervals their central open interval of length $\lambda(1-2\lambda)$. Continuing inductively, at step $k$ we will be left with $2^{k}$-many closed intervals of lengths $\lambda^{k}$. Let $\{I_{k,j}\}_{j=1}^{2^k}$ be the enumeration of them from left to right. Under this notation,
\begin{equation*}
    \textswab{C}(\lambda)=\bigcap_{k=0}^\infty \bigcup_{j=1}^{2^k} I_{k,j}.
\end{equation*}
The Cantor set $\textswab{C}(\lambda)$ has Hausdorff dimension ${\log 2}/{\log (1/\lambda)}$. We denote by $\mathcal{I}$, the collection of all the \emph{surviving} intervals, i.e.~ all $I_{k,j}$, where $k \in \{0,1,\ldots\}$, $j\in \{1,\ldots,2^{k}\}$. The intervals removed from $[0,1]$ in the process are called the \emph{removed} intervals.
\begin{remark}
    \label{minit1}
    To obtain step $k+1$ from step $k$ in the construction of $\textswab{C}(\lambda)$, $k\in \{0,1,\ldots\}$, we remove $2^{k}$-many open intervals of length $(1-2\lambda)\lambda^{k}$ from $I_{k,j}$'s. Moreover, between any two distinct $I_{k,j}$ and $I_{k,j'}$ there exists a removed interval of size larger than or equal to $(1-2\lambda)\lambda^{k-1}$.    
\end{remark}
Let $\textswab{D} \subset \mathbb R^1$ be a Cantor set, i.e.\ a totally disconnected compact set, which, for now, might not be self-similar. Then, there exists a countable collection of pairwise disjoint open intervals $J \subset [\inf \textswab{D}, \sup \textswab{D} ]$ such that 
\begin{equation}
\label{J-intvl}
    \textswab{D}=[\inf \textswab{D},\sup \textswab{D}]\setminus \bigcup J.
\end{equation}
Indeed, $J$ are the connected components of $[\inf \textswab{D},\sup \textswab{D}] \setminus \textswab{D}$. 

We denote by $\mathcal{J}$, the collection of \emph{the closures} of all $J$ from \eqref{J-intvl}. So, when we write $J \in \mathcal J$, it is understood that $J$ is closed. We assume that $\textswab{D}$ has no isolated points, hence, intervals from $\mathcal{J}$ are also pairwise disjoint.
\begin{remark}
Eventually, we will take $\textswab{D}$ to be a scaled copy of $\textswab{C}(\lambda'),$ for some $\lambda' \in (0, 1/2)$. Then $\mathcal J$ will be exactly the closures of the removed intervals in its construction (see Remark~\ref{minit1}).
\end{remark}

\subsection{$\mathbf{X}$ as a subset of the plane}
\label{asaset}
We fix $\textswab{C}=\textswab{C}(\lambda)$, $\lambda \in (0,1/2)$, and a second Cantor set $\textswab{D}$ as above. Recall the families of closed intervals $\mathcal{I}$ and $\mathcal{J}$ associated to $\textswab{C}$ and $\textswab{D}$, respectively. We wish to emphasize that $\mathcal{I}$ and $\mathcal{J}$ are chosen very differently and have opposite geometric interpretations; see the previous subsection for details. 
We assume that
\begin{align}\label{assum1}
      |J| < \lambda^{-1} \; \text{for all $J \in \mathcal{J}$, and} \; |J| \geq 1 \ \text{for at least one $J \in \mathcal{J}$.}
\end{align}
The latter condition will ensure the connectivity of  our space (Lemma \ref{qgd1}).
\begin{definition}
    \label{cin6}
For every $J \in \mathcal{J}$, fix the unique integer $k \in \{0,1,\ldots\}$ such that $\lambda^{k} \leq |J| < \lambda^{k-1}$. For $j=1,\ldots,2^{k}$ we call each of the closed rectangles $I_{k,j} \times {J}$ a cube. Row $J$ consists of the collection of all cubes $Q=I \times {J}$ for different $I$ (and same $J$). We say that the cube $I \times {J}$ is on row $J$, and
it is a cube over $I$.
\end{definition}
\begin{figure}[h]
    \begin{tikzpicture}[line width=0.5pt, scale=0.09]
  \fill[gray!30] (0,108) rectangle ++(64,10);
  \draw[line width=0.2pt] 
    (0,118) -- (64,118);
  \draw[line width=0.2pt] 
    (0,108) -- (0,118);
    \draw[line width=0.2pt] 
    (64,108) -- (64,118);

    \foreach \k in {126, 150} 
    \foreach \l in {0,48} \draw[line width=.2pt,fill=gray!30] (\l,\k) rectangle ++(16,16);

    \foreach \m in { 120, 144 } 
    \foreach \n in {0, 12, 48, 60} \draw[line width=.2pt,fill=gray!30] (\n,\m) rectangle ++(4,4);

    \foreach \m in {118.5, 124.5, 142.5, 148.5, 166.5} 
    \foreach \n in {0, 3, 12, 15, 48, 51, 60, 63} \draw[line width=.2pt,fill=gray!30] (\n,\m) rectangle ++(1,1);
    \end{tikzpicture}
\caption{Part of the space $\mathbf X$ (shaded).}
\label{fig:part-of-X}
\end{figure}
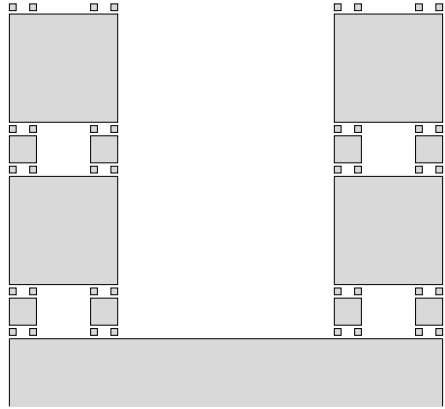

\begin{definition}[The space $\mathbf{X}$]
    \label{defX}
 We write $\mathbf{X}'$ for the union of all the cubes from Definition~\ref{cin6}, for all $J \in \mathcal J$. The space $\mathbf{X}$ is the closure  of $\mathbf{X}'$ in the Euclidean topology of $\mathbb R^2$. 
\end{definition}

\begin{remark}\label{rem-july-1}
Some observations about the structure of $\mathbf{X}$ that we shall later use without further comments:
\begin{enumerate} 
    \item For some $J \in \mathcal J$, the cube $[0,1]\times J$ is in $\mathbf X$.
    \item No two cubes (whether from the same or different rows) intersect.
    \item For any pair of cubes $Q=I\times J$ and $Q'=I'\times J'$, either $I \cap I' = \emptyset$ or $I\subset I'$ or $I' \subset I$.
    \item Suppose $J, J' \in \mathcal{J}$ with $|J| \leq |J'|$. Then for every cube $Q=I\times J$ there is a cube $Q'=I' \times J'$ such that $ I \subset I'$.
    \item Let $t \in \textswab{C} $. For every  $J \in \mathcal{J}$ there is a cube $Q=I\times J$ with $t \in I$.
\end{enumerate}    
\end{remark}
\begin{lemma}
\label{dan7}
    Suppose $p_0=(x_0,y_0) \in \mathbf{X} \setminus \mathbf{X}'$. Then,
    \begin{enumerate}[topsep=0.1em]
        \item[(a)] \label{lem3.5-a} \ $p_0 \in \textswab{C}  \times \textswab{D}$, consequently, $\mathbf{X}=\mathbf{X}' \cup (\textswab{C}  \times \textswab{D})$.
        \item[(b)] \label{lem3.5-b} There exists a sequence of cubes $Q_j$ and $p_j=(x_0,y_j) \in Q_j$ such that
        $\|p_j-p_0\| \to 0$.
        \item[(c)]\label{lem3.5-c} Given any sequence of cubes $Q_j$ and $p_j \in Q_j$, if $\|p_j-p_0\| \to 0$, then
        $\diam  Q_j \to 0 $.
    \end{enumerate}
\end{lemma}
\begin{remark}
    Observe however that countably many points from $\textswab{C}  \times \textswab{D}$ do belong to the cubes.
\end{remark}
\begin{proof}[Proof of Lemma \ref{dan7}]
We first prove (c).
Let $p_j\in Q_j$ satisfy $\|p_j-p_0\|\to 0$.
For any cube $Q$, only finitely many $p_j$ can satisfy $p_j \in Q$; otherwise a subsequence of $p_j$ would lie within $Q$ and since $Q$ is closed, its limit $p_0$ would belong to $Q$, hence to $\mathbf{X}'$, which is a contradiction.

Let $\varepsilon>0$ be arbitrary. 
Note that
there are only finitely many cubes $Q$ with $\diam  Q \geq \varepsilon$. By the preceding argument, there exists $N_0\in\mathbb N$ such that if $j > N_0$, then $p_j$ is in none of those cubes. Hence, for every $j>N_0,$ we have $\diam  Q_j < \varepsilon$. This proves $\diam  Q_j\to 0$ as $j\to\infty$.

We next prove claim (a). Since $\mathbf{X}$ is the closure of the union of all cubes, there is a sequence of cubes $Q_j$ and points $p_j:=(x_j,y_j) \in Q_j$ such that $\|p_j-p_0\|\to 0$. By claim (c), $\diam  Q_j\to 0$.  For each $j \in \mbbn$, choose some $(x_j',y_j) \in Q_j$ with $ x_j' \in \textswab{C}  $ (such points exist). It now follows that
    \begin{align*}
    |x_j' - x_0| & \leq |x_j'-x_j| + |x_j - x_0| \leq \diam  Q_j+\|p_j-p_0\|.        
    \end{align*}
Both terms on the right converge to zero, hence $x_j' \to x_0$. Since $x_j' \in \textswab{C} $ for all $j$ and since $\textswab{C}$ is closed, we conclude that $x_0 \in \textswab{C}$. Now, observe that if $y_0$ were not in $\textswab{D}$, then $y_0$ would be in some $J \in \mathcal{J}$, and so, $p_0=(x_0,y_0)$ would belong to a cube on row $J$; which contradicts the assumption that $p_0 \notin \mathbf{X}'$. Thus, $y_0 \in \textswab{D}$ and we conclude that $p_0\in \textswab{C}\times \textswab{D}$.

Finally, we prove claim (b). By properties of the Cantor set $\textswab{D}$, there exists a sequence of $J_j \in \mathcal{J}$ and $y_j \in J_j$, such that $y_j \to y$ as $j \to \infty$. Clearly, each $p_j:=(x_0,y_j)$ belongs to a cube on row $J$ and $|y_j-y_0| = \|p_j-p_0\|\to 0$. The proof is complete.
\end{proof}

\begin{lemma}\label{lem3.9}
    $\mathbf{X}$ contains the set $\, \textswab{C} \times [\inf \textswab{D},\sup \textswab{D}]  $.
\end{lemma}
\begin{proof}
    Fix an arbitrary $t \in \textswab{C}$. For every $J \in \mathcal{J}$, there exists a cube $Q=I \times J$ with $t \in I$ and hence 
    $\{t\} \times J \subset \textswab{C}\times [\inf \textswab{D},\sup \textswab{D}]$.
    Since $[\inf \textswab{D},\sup \textswab{D}]$ is the closure of the union of all $J$, and 
    $\mathbf X$
    is closed in the plane, it follows that $\{t\} \times [\inf \textswab{D},\sup \textswab{D}]$ lies in $\mathbf{X}$. Since $t \in \textswab{C}$ was arbitrary, the proof is complete.
\end{proof}
\subsection{Metric and measure properties of $\mathbf{X}$}
The set $\mathbf X$ inherits the Euclidean metric from $\mathbb R^2$. We equip $\mathbf{X}$ with the Lebesgue measure induced from $\mathbb R^2$, which we will denote by $\mu$. As $(\mathbf{X},\|\cdot\|)$ is compact in $\mathbb R^2$, a set $A\subset \mathbf X$ is $\mu$-measurable if and only if $A$ is Lebesgue measurable as a subset of $\mathbb R^2$.

We write $\ell(\gamma)$ for the length of (rectifiable) curves, with respect to this metric.
\begin{lemma}
    \label{qgd1}
    For every $p,q\in \mathbf{X}$ there exists a path $\gamma \colon [0,1] \to (\mathbf X, \|\cdot\|)$ with $\gamma(0)=p$, $\gamma(1)=q$, and $ \ell(\gamma) < \infty $.
\end{lemma}
\begin{proof}
Indeed, by Lemma~\ref{lem3.9} and by the assumption \eqref{assum1} that $|J| \geq 1$ for some $J$, there exists such a path consisting of finitely many horizontal and vertical segments. 
\end{proof}
We  additionally assume two key conditions on $\textswab{D}$, stated in terms of its complementary intervals.
\begin{definition}[$\rho$-separated]\label{uniform-separattion}
   Let $0<\rho<1$.
    We say that $\mathcal J$ is $\rho$-\textit{separated} if for every distinct intervals $J, J' \in \mathcal J$ with $\min\{|J|,|J'|\} \ge \ell$, we have $\dist(J,J') \ge \rho \, \ell$.
\end{definition}
\begin{definition}[$\tau$-dense]
    \label{def:density}
    Let $\tau\geq 1$.
    We say that $\mathcal {J}$ is \emph{$\tau$-dense} if for every $r \in (0,1]$ and every $y \in [\inf \textswab{D},\sup \textswab{D}]$, there exists $J \in \mathcal{J}$ such that  $|J| \geq r$ and the distance of $J$ to $y$ is at most $\tau r$.
\end{definition}
The two conditions are complimentary to each other: $\rho$-separatedness guarantees that large intervals are not too close to each other, whereas $\tau$-density ensures that large intervals are not too sparse.

For the self-similar Cantor sets both properties hold. The proof easily follows from the construction; see Remark~\ref{minit1}.
\begin{lemma}\label{prop3.16}
    Suppose $\, \textswab{D}=\frac{1}{1-2\lambda}\textswab{C}(\lambda)$, $\lambda \in (0,1/2)$. Then $\mathcal{J}$ is $\rho$-separated and $\tau$-dense for some $\tau$ and $\rho$ that (separately) depend only on $\lambda$.
\end{lemma}

\subsubsection{Quasiconvexity} We define \emph{the length metric} on $\mathbf X$ by
    $$
    d(p,q):=\inf \{\ell(\gamma) : \text{$\gamma$ is a path in $(\mathbf X, \|\cdot\|)$ that joins $p$ to $q$} \}.
    $$
Lemma \ref{qgd1} shows that indeed $(\mathbf{X},d)$ is a metric space. 

Clearly, $d(p,q) \geq \|p-q\|$ for all points and $d(p,q) = \|p-q\|$ if the two points belong to the same cube. But a quantitative reverse relation holds, i.e.\ $(\mathbf X,\|\cdot\|)$ is so-called \emph{quasiconvex} (cf.\ \cite[Proposition~4.10]{Kosk-wild}).
\begin{proposition}
\label{prop:bilip-eqiv}
If in the construction of $\mathbf X$, corresponding to $\textswab{C}(\lambda)$ and $\textswab{D}$, the removed intervals for $\textswab{D}$ are $\rho$-separated and $\tau$-dense
for some $\lambda \in (0,1/2), \tau\geq 1, 0<\rho<1$,
then the identity map $(\mathbf X,d) \to (\mathbf X,\|\cdot\|)$ is bi-Lipschitz, with a constant that depends only on $\lambda$, $\rho$ and $\tau$.
\end{proposition}
The idea is that to join a pair of points between cubes that do not align vertically, we first move ``vertically'', from each one, until we reach a common large cube that allows to join horizontally. The separatedness and density conditions ensure that we can do this within a reasonable distance from the original points. In Section~\ref{sec:curve-fam} we will construct much more: a family of curves between any two cubes. So, we do not provide a separate proof here.

It is a well-known fact that if we equip a proper quasiconvex space with the length metric, then it becomes geodesic; see e.g.~\cite[Lemmas~8.3.11 and 8.3.12]{HKST:15}.
\begin{corollary}
     The metric space $(\mathbf{X},d)$ is geodesic, and the identity map $(\mathbf X,\|\cdot\|) \to (\mathbf{X},d)$ is bi-Lipschitz. Moreover, the class of rectifiable curves and the length of individual curves are the same with respect to either metric.  
\end{corollary}
\textit{In what follows, the default metric on $\mathbf{X}$ will be the intrinsic metric $d$, with respect to which $\mathbf X$ is geodesic.}

\subsubsection{Ahlfors $2$-regularity} Recall that $\mu$ is the restriction of the Lebesgue measure on $\mathbb R^2$ to $\mathbf X$.
\begin{proposition}
    \label{cin9}
Suppose that $\mathcal J$ is $\tau$-dense, $\tau\geq 1$, in the sense of Definition~\ref{def:density}.
Then $(\mathbf{X},d,\mu)$ is Ahlfors 2-regular, with constants that depend only on $\lambda$ and $\tau$.
\end{proposition}

\begin{proof}
Let $p=(x,y)\in X$ and $0<r<\diam \mathbf X$. Since $\|p-q\|\leq d(p,q)$ for  all points $q\in X$, we have $B(p,r) \subset \{q: \|q-p\|<r\}$. Hence, the upper bound $\mu(B(p,r)) \leq \pi r^2$ is always true. So, it remains to show
    \begin{equation}\label{3.3-june}
   \mu(B(p,r)) \gtrsim r^2,
    \end{equation}
with a constant that does not depend on $r, p$.

We choose the (not so significant) cutoff of $ 2+\tau $. By applying the density condition with $r=1$, we see that \eqref{3.3-june} holds for all $ 2+\tau \le r < \diam \mathbf X$ (compare to case (2) below). Hence, we shall restrict to the case of $r\le 2+\tau$. First, let us assume that $p\in X'$. Let $Q=I\times J$ be the cube that contains $p=(x,y)$. We obtain \eqref{3.3-june} by considering two cases.
    
{Case (1):} When $r/(2+\tau) \le |J|$, we have a uniform quantitative bound, in terms of $\lambda$ and $\tau$, on the measure of the portion of $B(p,r)$ that lies inside $Q=I\times J$. So, in this case, $\mu(B(p,r)) \gtrsim r^2$ holds with a constant that depends only on $\lambda$ and $\tau$.

{Case (2)}: When $|J| < r/(2+\tau)$, we find a lower bound on $\mu(B(p,r))$ as follows. Write $r':=r/(2+\tau)$, which is less than $1$. We fix $t \in \textswab{C}\cap I$ and consider $p_1=(t,y)$. By the density of $\mathcal J$, we find $J'\in \mathcal{J}$
such that $|J'| \ge r'$ and $|y-y_0| \le \tau r'$ for some $y_0 \in J'$. Let $p_2:=(t,y_0)$.
    
 Because $d(p,p_1) \le |I|<r'$ and $d(p_2,p_1)\le \tau r'$, from the triangle inequality we get
    \[
    B(p_2,r') \subset B(p_1,(\tau+1)r') \subset B(p,r).
    \]
Since $r'\le |J'|$, {case (1)} applied to the point $p_2$ and the radius $r'$ yields 
$$
\mu(B(p,r)) \ge \mu(B(p_2,r')) \gtrsim r^2.
$$
We have proved \eqref{3.3-june} for $p \in \mathbf X'$. To prove \eqref{3.3-june} for $p \in \mathbf X \setminus \mathbf X'$, we apply Lemma ~\ref{dan7} to choose a sequence $p_j \in \mathbf X'$ such that $d(p_j,p) \to 0$. For every $\eps \in (0, r/2)$ there exists $j$ such that $B(p_j,r-\eps) \subset B(p,r)$. Applying the previous analysis to $p_j\in\mathbf X'$ gives
    \[
    \mu(B(p,r)) \geq \mu(B(p_j,r-\eps)) \gtrsim (r-\eps)^2,
    \]
with the constants independent of $\varepsilon$. By letting $\varepsilon\to 0$, we obtain the desired estimate $\mu(B(p,r)) \gtrsim r^2$. The proof is complete.
\end{proof}
\subsection{Comparison to the construction in~{\bf \cite{Kosk-wild}}}\label{sec:bilip-kosk-wil}
First of all one must rotate the construction in \cite{Kosk-wild} by $90$ degrees for the comparisons below to apply. Let us denote the space constructed in \cite{Kosk-wild} by $\widetilde{X}$. Under our notation, $ \widetilde{X} $ corresponds to our $\textbf{X}$ with the choice of $\textswab{C}=\textswab{C}(3^{-m})$ and $\textswab{D}=\textswab{C}(3^{-n})$, where $1 \le m <n $ are integers such that $\frac{n}{m} \in \{2,3,\ldots\}$. Notice that $\textswab{D}$ is not scaled here.

The surviving intervals $I$ for $\textswab{C}$ have lengths of the form $3^{-im}$ and the removed intervals $\mathcal J$ for $\textswab{D}$ have lengths of the form $3^{-i'n}(3^n-2)$; see Remark~\ref{minit1}. Therefore, very importantly,
\begin{equation}\label{ratio-never-1}
    \text{the ratios $\frac{|I|}{|J|}$ are bounded away from $1$}
\end{equation}
by a quantity that depends only on $n$ and $m$.

For fixed $J$, the following set will be in $\widetilde{X}$, except possibly for
the few largest $J$,
\begin{equation}\label{tilde-pieces}
    \{(x,y): \inf J + \dist(x,\textswab{C}) \le y \le \sup J - \dist(x,\textswab{C}) \}.
\end{equation}
Observe that depending on $|J|$ this consists of multiple connected components that are congruent to one another. They will be bounded by vertical rigid shifts of the graphs of $x \mapsto \pm \dist(x,\textswab{C})$ (Figure~\ref{fig:kosk-wild}).

\begin{figure}[h]
    \begin{tikzpicture}[scale=3]
\foreach \m in {0,2} 
{
    \draw[line width=.2pt] (\m+1,0) -- (\m+1+1/3,1/3);
    \draw[line width=.2pt] (\m,0) -- (\m-1/3,1/3);
    \draw[line width=.2pt] (\m+1,2/3) -- (\m+1+1/3,2/3-1/3);
    \draw[line width=.2pt] (\m,2/3) -- (\m-1/3,2/3-1/3);
}

\foreach \m in {0,2} 
{
    \draw[line width=.2pt] (\m+1/3,0) -- (\m+1/3+0.5*1/3,1/6);
    \draw[line width=.2pt] (\m+1/3+0.5*1/3,1/6) -- (\m+1/3+1/3,0);
}

\foreach \m in {0,2/3, 2, 2+2/3} 
\foreach \i in {1/9} 
\draw[line width=.2pt] (\m+\i,0) -- (\m+\i+0.5*1/9,1/18) -- (\m+\i+1/9,0);

\foreach \m in {0,2/9,2/3,8/9,  2,2+2/9,2+2/3,2+8/9}
\foreach \i in {1/27} 
\draw[line width=.2pt] (\m+\i,0) -- (\m+\i+0.5*1/27,1/54) -- (\m+\i+1/27,0);

\foreach \m in {0,2/27,6/27,8/27, 18/27, 20/27, 24/27,26/27, 2,2+2/27,2+6/27,2+8/27,2+18/27,2+20/27,2+24/27,2+26/27}
\foreach \i in {1/81} 
\draw[line width=.2pt] (\m+\i,0) -- (\m+\i+0.5*1/81,1/162) -- (\m+\i+1/81,0);

\foreach \y in {2/3}
\foreach \m in {0,2} 
{
    \draw[line width=.2pt] (\m+1/3,\y) -- (\m+1/3+0.5*1/3,\y-1/6);
    \draw[line width=.2pt] (\m+1/3+0.5*1/3,\y-1/6) -- (\m+1/3+1/3,\y);
}

\foreach \y in {2/3}
\foreach \m in {0,2/3, 2, 2+2/3} 
\foreach \i in {1/9} 
\draw[line width=.2pt] (\m+\i,\y) -- (\m+\i+0.5*1/9,\y-1/18) -- (\m+\i+1/9,\y);

\foreach \y in {2/3}
\foreach \m in {0,2/9,2/3,8/9,  2,2+2/9,2+2/3,2+8/9}
\foreach \i in {1/27} 
\draw[line width=.2pt] (\m+\i,\y) -- (\m+\i+0.5*1/27,\y-1/54) -- (\m+\i+1/27,\y);

\foreach \y in {2/3}
\foreach \m in {0,2/27,6/27,8/27, 18/27, 20/27, 24/27,26/27, 2,2+2/27,2+6/27,2+8/27,2+18/27,2+20/27,2+24/27,2+26/27}
\foreach \i in {1/81} 
\draw[line width=.2pt] (\m+\i,\y) -- (\m+\i+0.5*1/81,\y-1/162) -- (\m+\i+1/81,\y);

\end{tikzpicture}
\caption{The space in \cite{Kosk-wild} is, practically, the union of the interiors of regions bounded between vertical shifts of the graphs of $x \to \pm \dist(x,\textswab{C})$.}
\label{fig:kosk-wild}
\end{figure}

The projection of each one of such pieces onto the $x$-axis is equal to 
$$
\Bigl[\inf I -\frac{|J|}{2},\sup I + \frac{|J|}{2}\Bigr],
$$
where $I$ is a surviving interval with a maximal length $|I|$ subject to $|I|\le |J|$.

Due to \eqref{ratio-never-1}, each of these pieces is bi-Lipschitz equivalent, with constant depending only on $n$ and $m$, to the rectangle $Q=I\times J$. (There is a quantitative lower bound on how close the top and bottom ``teeth'' can get to each other, after normalizing for scale.)

As in our case, one needs a large piece to ensure the connectivity of the space. We do so in this paper by scaling $\textswab{D}$ such that at least one $J$ has $|J| \ge 1$. Instead, in \cite{Kosk-wild} they ``manually'' attach one large piece to one end of the union of the pieces above. Finally, like we do, they add $\textswab{C} \times \textswab{D}$ to complete the space.

In \cite{Kosk-wild} they prove that $\widetilde{X}$, equipped with the standard Euclidean distance and measure, is quasiconvex (\cite[Proposition~4.10]{Kosk-wild}) and Ahlfors $2$-regular (\cite[Proposition~4.11]{Kosk-wild}). They use $\widetilde{X}$ to prove the necessity of a $p$-Poincar\'e inequality in the second claim of Theorem~\ref{bakoro}. They conjecture that $\widetilde{X}$ might support some weaker Poincar\'e inequality but still better than a $2$-Poincar\'e inequality.

Although the construction in \cite{Kosk-wild} is the inspiration for our $\mathbf X$ and very similar in nature, our construction is much less technical, which is why we are able to prove sharp Poincar\'e inequalities.

We just mentioned that the individual building blocks of $\widetilde{X}$ are uniformly bi-Lipschitz to the building blocks (the cubes) in our construction of $\mathbf{X}$. The relative positioning of these blocks in the two spaces are also very similar (but one has to consider only ``half of'' our $\mathbf{X}$ for this). Thus, by some more analysis one can show that the spaces are (globally) bi-Lipschitz. We do not carry out the details for the same reason that calculations become unruly.

Therefore, by proving a $p$-Poincar\'e inequality for $\mathbf X$, we indirectly confirm that the space $\widetilde{X}$ in \cite{Kosk-wild} satisfies a $p$-Poincar\'e inequality for the same range of $p$; thus confirming their conjecture.
 
\section{Path Families in \protect\boldmath {$\mathbf{X}$}}
\label{sec:curve-fam}
In this section we construct path families that will act as the building blocks of the more complicated ones used later in the proof of the Poincar\'e inequality on $\textbf{X}$. We continue with the notation set forth in the previous sections. In particular, $\textswab{C}=\textswab{C}(\lambda)$, for some $\lambda \in (0, 1/2)$, and $\mathcal {J}$, the collection of the complementary intervals of $\textswab{D}$, is dense and separated in the sense of Definitions~\ref{uniform-separattion} and ~\ref{def:density}.

We write $\lambda^*:=\frac{\lambda}{2(1-\lambda)}$, observe that $2\lambda^*=\sum_{j=1}^\infty \lambda^j$. Also set $I_{01}=[0,1]$; and $Q_{I_{01}}:=[0,1] \times [0,\lambda^*]$.

\subsection{The path family \protect\boldmath {$\Gamma(I)$}}
\label{sec-pathQ}
 First consider $I=I_{01}=[0,1]$. The construction of $\Gamma(I)$ is self-similar and quite intuitive (Figure~\ref{pic-gam0}).

 \begin{figure}[h]
    {\centering
    \begin{tikzpicture}[line cap=round,line join=round,scale=0.2]
  \draw[thick] (0,0) rectangle (32,16);
    \draw[dashed] (0,4) -- (8,4);
     \draw[dashed] (24,4) -- (32,4);
    \draw[dashed] (0,2) -- (2,2);
    \draw[dashed] (6,2) -- (8,2);
    \draw[dashed] (24,2) -- (26,2);
    \draw[dashed] (30,2) -- (32,2);
  \foreach \x in {0,0.5,1,1.5,2,2.5,3,3.5,4,4.5,5,5.5,6,6.5,7,7.5,8,8.5,9,9.5,10,10.5,11,11.5,12,12.5,13,13.5,14,14.5,15,15.5,16}
  {
    \draw (0.5*\x,4) -- (\x,16);
    \draw (32-0.5*\x,4) -- (32-\x,16);
    }

\foreach \x in {0,1,2,3,4,5,6,7,8,9,10,11,12,13,14,15,16}
{
    \draw (0.125*\x,2) -- (0.25*\x,4);
    \draw (8-0.125*\x,2) -- (8-0.25*\x,4);
}

\foreach \x in {0,1,2,3,4,5,6,7,8,9,10,11,12,13,14,15,16}
{
    \draw (24+0.125*\x,2) -- (24+0.25*\x,4);
    \draw (32-0.125*\x,2) -- (32-0.25*\x,4);
}

\foreach \x in {0,2,4,6,8,10,12,14,16}
{
    \draw (0.03125*\x,1) -- (0.0625*\x,2);
    \draw (6+0.03125*\x,1) -- (6+0.0625*\x,2);
    \draw (24+0.03125*\x,1) -- (24+0.0625*\x,2);
    \draw (30+0.03125*\x,1) -- (30+0.0625*\x,2);
    \draw (2-0.03125*\x,1) -- (2-0.0625*\x,2); 
    \draw (8-0.03125*\x,1) -- (8-0.0625*\x,2);
    \draw (26-0.03125*\x,1) -- (26-0.0625*\x,2);
    \draw (32-0.03125*\x,1) -- (32-0.0625*\x,2);
}
\foreach \i in {0.3,1.8,6.3,7.8,24.3,25.8,30.3,31.8}
\foreach \k in {0.1,0.4,0.7}
{
    \fill (\i,.1) ++(0,\k) circle[radius=0.05];
  }
\node at (-2.2,10) {$\frac{\lambda}{2}$};
\node at (-2.2,3) {$\frac{\lambda^2}{2}$};
\node at (34,10) {$\lambda^*$};
\end{tikzpicture}
    \caption{The path family $\Gamma(I_{01})$.}
    }
    \label{pic-gam0}
\end{figure}
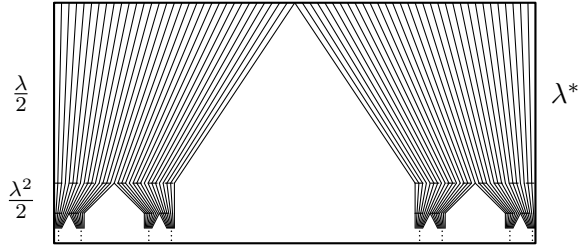

We begin with a copy of $[0,1]$ at height $\lambda^*$. We split it into two halves. For $t \in [0,1/2)$, the path $\gamma_t$ that we wish to construct begins as the straight line that joins $(t,\lambda^*)$ to $(b(t),\lambda^*-\lambda/2)$, where $b(t)$ is the image of $t$ under the canonical bijection between $[0,1/2]$ and $[0,\lambda]$. For $t \in (1/2,1]$, $\gamma_t$ joins $(t,\lambda^*)$ to $(b(t),\lambda^*-\lambda/2)$ where now $b$ is the canonical bijection between $[1/2,1]$ and $[1-\lambda,1]$.

We do a similar split on each of the intervals of length $\lambda$ that contain the endpoints of $\gamma_t$, currently at height $\lambda^*-\lambda/2$. This time, we drop in height by $\lambda^2/2$.

We repeat inductively ad infinitum. At each step, we do not assign a path to the midpoints of the intervals. Note that these piecewise linear paths, at all stages, deviate from the vertical direction by a controlled angle that depends only on $\lambda$.

By parameterizing carefully and concatenating the pieces, for every $t \in [0,1]$, with the exception of a countable set of ``middle points'', we associate a path $\gamma_t$ that is parameterized over $(0,\lambda^*]$ and the Lipschitz constant of $\gamma_t$ depends only on $\lambda$. Thus 
each $\gamma_t$ has a unique extension to $[0,\lambda^*]$, with the same Lipschitz constant. By construction, we have $\gamma_t(0) \in \textswab{C}(\lambda) \times \{0\}$.

We denote by $\Gamma(I_{01})$ the collection of all these $\gamma_t$. The exceptional set of $t$ will never be important as we will integrate in $t$ with respect to the Lebesgue measure, so, we think of $\Gamma(I_{01})$ as being indexed over $[0,1]$, and might not point out the exceptional set in future instances.

\begin{definition}
Given $I \in \mathcal I$, we scale by the factor $|I|$ (and shift) every aspect of the construction of the path family $\Gamma(I_{01})$ to obtain the path family $\Gamma(I)$, indexed over $I$. The Lipschitz constant for the paths remains the same as for $\Gamma(I_{01})$, hence, depends only on $\lambda$. When $Q=I\times J$, we use the same notation $\Gamma(I)$ for its (possibly upside-down) copies that lie along the top and bottom edges of $Q$.
\end{definition}

\begin{remark}
    It may look counter-intuitive, but for every Cantor point there is exactly one curve that lands on it; except for countably many Cantor points that receive no curves at all. So, we could think of the path families as being indexed over $\textswab{C}$, but we will not utilize this point of view. Figure~\ref{curtains} can be helpful here.
\end{remark}
The following is immediate from the construction.
\begin{lemma}[Self-similarity in $\Gamma(I)$]
\label{gammacopy}
    Suppose $I,I'$ in $\mathcal{I}$, with $I' \subset I$. For every $t' \in I'$, $\gamma_{t'} \in \Gamma(I')$ is a subpath of $\gamma_{t} \in \Gamma(I)$ for a unique $t \in I$. Both $\gamma_{t'}$ and $\gamma_t$ end on the same point on $\textswab{C} \cap I'$. Moreover, $t' \mapsto t$ is a canonical bijection from $I'$ to a subinterval of $I$.
\end{lemma}
It will be very important that every path in $\Gamma(I)$ is uniquely determined by the point on the Cantor set that it passes through.

\subsubsection{The path family \protect\boldmath {$\Gamma(Q)$}}
Recall that a cube in our space $\mathbf X$ is a rectangle $Q=I \times J$, where $ |I| \leq |J| < \lambda^{-1}|I|$. Because $\lambda^* <1/2$, we can place a copy of $\Gamma(I)$ at the bottom edge of $Q$ and an upside-down copy of $\Gamma(I)$ at the top edge of $Q$. This will still leave a rectangular blank piece along the horizontal mid segment of $Q$. By attaching honest vertical paths here we can concatenate the two copies of $\Gamma(I)$ and have a family of paths that join the copy of the Cantor set $\textswab{C}$ along the top edge of $Q$ to the copy of the Cantor set $\textswab{C}$ along the bottom edge of $Q$. (Figure \ref{pic-gamQ}).

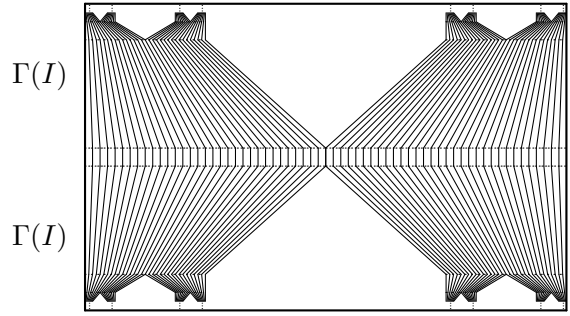
\begin{figure}
    \centering
\begin{tikzpicture}[line cap=round,line join=round,xscale=0.2,yscale=0.12]
\begin{scope}[yscale=-1, shift={(0,-34)}]
    \draw[thin, dotted] (0,4) -- (8,4);
     \draw[thin, dotted] (24,4) -- (32,4);
    \draw[thin, dotted] (0,2) -- (2,2);
    \draw[thin, dotted] (6,2) -- (8,2);
    \draw[thin, dotted] (24,2) -- (26,2);
    \draw[thin, dotted] (30,2) -- (32,2);
  \foreach \x in {0,0.5,1,1.5,2,2.5,3,3.5,4,4.5,5,5.5,6,6.5,7,7.5,8,8.5,9,9.5,10,10.5,11,11.5,12,12.5,13,13.5,14,14.5,15,15.5,16}
  {
    \draw (0.5*\x,4) -- (\x,16);
    \draw (32-0.5*\x,4) -- (32-\x,16);
    }

\foreach \x in {0,1,2,3,4,5,6,7,8,9,10,11,12,13,14,15,16}
{
    \draw (0.125*\x,2) -- (0.25*\x,4);
    \draw (8-0.125*\x,2) -- (8-0.25*\x,4);
}

\foreach \x in {0,1,2,3,4,5,6,7,8,9,10,11,12,13,14,15,16}
{
    \draw (24+0.125*\x,2) -- (24+0.25*\x,4);
    \draw (32-0.125*\x,2) -- (32-0.25*\x,4);
}

\foreach \x in {0,2,4,6,8,10,12,14,16}
{
    \draw (0.03125*\x,1) -- (0.0625*\x,2);
    \draw (6+0.03125*\x,1) -- (6+0.0625*\x,2);
    \draw (24+0.03125*\x,1) -- (24+0.0625*\x,2);
    \draw (30+0.03125*\x,1) -- (30+0.0625*\x,2);
    \draw (2-0.03125*\x,1) -- (2-0.0625*\x,2); 
    \draw (8-0.03125*\x,1) -- (8-0.0625*\x,2);
    \draw (26-0.03125*\x,1) -- (26-0.0625*\x,2);
    \draw (32-0.03125*\x,1) -- (32-0.0625*\x,2);
}
\foreach \i in {0.3,1.8,6.3,7.8,24.3,25.8,30.3,31.8}
\foreach \k in {0.1,0.4,0.7}
{
    \fill (\i,.1) ++(0,\k) circle[radius=0.05];
  }
  \end{scope}
    \draw[thin, dotted] (0,4) -- (8,4);
     \draw[thin, dotted] (24,4) -- (32,4);
    \draw[thin, dotted] (0,2) -- (2,2);
    \draw[thin, dotted] (6,2) -- (8,2);
    \draw[thin, dotted] (24,2) -- (26,2);
    \draw[thin, dotted] (30,2) -- (32,2);
  \foreach \x in {0,0.5,1,1.5,2,2.5,3,3.5,4,4.5,5,5.5,6,6.5,7,7.5,8,8.5,9,9.5,10,10.5,11,11.5,12,12.5,13,13.5,14,14.5,15,15.5,16}
  {
    \draw (0.5*\x,4) -- (\x,16);
    \draw (32-0.5*\x,4) -- (32-\x,16);
    }

\foreach \x in {0,1,2,3,4,5,6,7,8,9,10,11,12,13,14,15,16}
{
    \draw (0.125*\x,2) -- (0.25*\x,4);
    \draw (8-0.125*\x,2) -- (8-0.25*\x,4);
}

\foreach \x in {0,1,2,3,4,5,6,7,8,9,10,11,12,13,14,15,16}
{
    \draw (24+0.125*\x,2) -- (24+0.25*\x,4);
    \draw (32-0.125*\x,2) -- (32-0.25*\x,4);
}

\foreach \x in {0,2,4,6,8,10,12,14,16}
{
    \draw (0.03125*\x,1) -- (0.0625*\x,2);
    \draw (6+0.03125*\x,1) -- (6+0.0625*\x,2);
    \draw (24+0.03125*\x,1) -- (24+0.0625*\x,2);
    \draw (30+0.03125*\x,1) -- (30+0.0625*\x,2);
    \draw (2-0.03125*\x,1) -- (2-0.0625*\x,2); 
    \draw (8-0.03125*\x,1) -- (8-0.0625*\x,2);
    \draw (26-0.03125*\x,1) -- (26-0.0625*\x,2);
    \draw (32-0.03125*\x,1) -- (32-0.0625*\x,2);
}
\foreach \i in {0.3,1.8,6.3,7.8,24.3,25.8,30.3,31.8}
\foreach \k in {0.1,0.4,0.7}
{
    \fill (\i,.1) ++(0,\k) circle[radius=0.05];
}

\foreach \x in {0,0.5,1,1.5,2,2.5,3,3.5,4,4.5,5,5.5,6,6.5,7,7.5,8,8.5,9,9.5,10,10.5,11,11.5,12,12.5,13,13.5,14,14.5,15,15.5,16}
{
    \draw(\x,16) -- (\x,18);
    \draw (32-\x,16) -- (32-\x,18);
    }
\draw[thick] (0,0) rectangle (0,34);
\draw[thick] (32,0) rectangle (32,34);
\draw[thick] (0,0) rectangle (32,0);
\draw[thick] (0,34) rectangle (32,34);
\draw[thin, dotted] (0,16) rectangle (32,16);
\draw[thin, dotted] (0,18) rectangle (32,18);

\node at (-3,26) {$\Gamma(I)$};
\node at (-3,8) {$\Gamma(I)$};
\end{tikzpicture}
    \caption{The path family $\Gamma(Q)$.}
    \label{pic-gamQ}
\end{figure}

We parameterize each path on the interval $J$ such that they have the same Lipschitz constant as the paths in $\Gamma(I)$. In particular, this Lipschitz constant depends only on $\lambda$. Observe also that $\gamma_t(\inf J)$ and $\gamma_t(\sup J)$ have the same first coordinate, which belongs to $\textswab{C} \cap I$; they join ``the same'' Cantor point on the two edges. We denote this path family by $\Gamma(Q)$.

By abuse of notation we will use $\gamma_t$ both for paths in $\Gamma(I)$ and in $\Gamma(Q)$.

\subsection{The path Family \protect\boldmath {$\Gamma({Q_1,Q_2})$}} 
\label{secjo1} These path families, and the subsequent integral estimates on them, will be significant in what follows as they ``encode'' all the infinitesimal scale analysis in $\mathbf X$.

We say that $Q_2=I_2 \times J_2$ is the first large cube below $Q_1=I_1 \times J_1$ if all of the following are satisfied (Figure~\ref{path-Q1-Q2}):
\begin{itemize}
    \item $\sup J_2 < \inf J_1$,
    \item $I_1 \subset I_2$, and
    \item if $Q=I\times J$ for some $J \subset (\sup J_2, \inf J_1)$, then $|I|\le |I_1|$.
\end{itemize}
Analogously, we can talk about the first large cube \emph{above} $Q_1$.

We set some notation:
    $$
    \mathcal{Q}(Q_1,Q_2):=\{Q=I\times J: J \subset (\sup J_2, \inf J_1) \; \text{and} \; I \subset I_1 \},
    $$
    $$
    \mathcal{J}(Q_1,Q_2):= \{J \in \mathcal{J}: J \subset (\sup J_2, \inf J_1)\}.
    $$

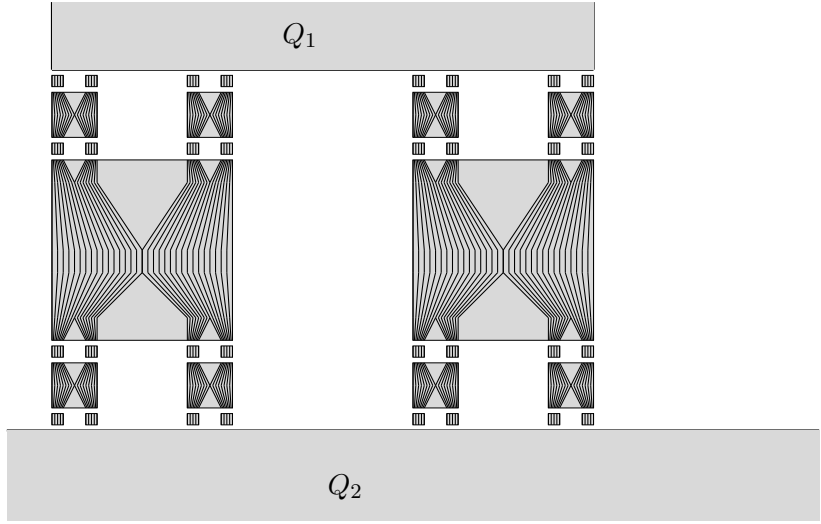
\begin{figure}[h] 
   \centering
\begin{tikzpicture}[scale=0.15]
  \draw[thick] (-4,118) -- ++(72,0);
  \fill[gray!30] (-4,110) rectangle ++(72,8); 

  \draw[thick] (0,150) -- (48,150);
  \draw[thick] (0,150) -- (0,156);
  \draw[thick] (48,150) -- (48,156);
  \fill[gray!30] (0,150) rectangle ++(48,6); 
    
    \foreach \k in {126} 
    \foreach \l in {0, 32} 
    \draw[line width=.2pt,fill=gray!30] (\l,\k) rectangle ++(16,16);

    \foreach \m in { 120, 144 } 
    \foreach \n in {0, 12, 32, 44} \draw[line width=.2pt,fill=gray!30] (\n,\m) rectangle ++(4,4);

    \foreach \m in {118.5, 124.5, 142.5, 148.5 } 
    \foreach \n in {0, 3, 12, 15, 32, 35, 44, 47} \draw[line width=.2pt,fill=gray!30] (\n,\m) rectangle ++(1,1);

\foreach \z in {0,32}
{
  \foreach \x in {0,1,2,3,4,5,6,7,8,9,10,11,12,13,14,15,16}
  {
    \draw[thin] (\z+0.5*\x,132) -- (\z+0.25*\x,128);
    \draw[thin] (\z+16-0.5*\x,132) -- (\z+16-0.25*\x,128);
    \draw[thin] (\z+0.5*\x,134) -- (\z+0.25*\x,140);
    \draw[thin] (\z+16-0.5*\x,134) -- (\z+16-0.25*\x,140);
     \draw[thin] (\z+0.5*\x,132) -- (\z+0.5*\x,134);
    \draw[thin] (\z+16-0.5*\x,132) -- (\z+16-0.5*\x,134);
}
}

\foreach \z in {0,32}
{
\foreach \y in {122, 140, 146}
{
  \foreach \x in {0,1,2,3,4,5,6,7,8}
  {
   \draw[thin] (\z+0.25*\x,\y) -- (\z+0.125*\x,\y+2);
 \draw[thin] (\z+4-0.25*\x,\y) -- (\z+4-0.125*\x,\y+2);
  \draw[thin] (\z+12+0.25*\x,\y) -- (\z+12+0.125*\x,\y+2);
    \draw[thin] (\z+16-.25*\x,\y) -- (\z+16-0.125*\x,\y+2);
}
}
}

\foreach \z in {0,32}
{
\foreach \y in {144, 126, 120}
{
  \foreach \x in {0,1,2,3,4,5,6,7,8}
  {
       \draw[thin] (\z+12+0.125*\x,\y) -- (\z+12+0.25*\x,\y+2);
        \draw[thin] (\z+16-0.125*\x,\y) -- (\z+16-0.25*\x,\y+2);
        \draw[thin] (\z+0.125*\x,\y) -- (\z+0.25*\x,\y+2);
        \draw[thin] (\z+4-0.125*\x,\y) -- (\z+4-0.25*\x,\y+2);
}
}
}

\foreach \w in {0,32}
{
\foreach \z in {0,3,12,15}
{
\foreach \y in {148.5, 142.5, 124.5,118.5}
{
\foreach \x in {0,2,4,6,8}
{
\draw[thin] (\w+\z+0.125*\x,\y) -- (\w+\z+0.125*\x,\y+1);
  }
  }
  }
  }

\node at (26,113) {\scalebox{0.99}{$Q_2$}};
\node at (22,153) {\scalebox{0.99}{$Q_1$}};
\end{tikzpicture}
 \caption{The path family $\Gamma(Q_1,Q_2)$. Clearly, some fine features of the paths within individual cubes are not present in the figure.}
    \label{path-Q1-Q2}
\end{figure}

Recall that for every $t \in I_1$, with the usual exception of countably many $t$, the path $\gamma_t \in \Gamma(I_1)$ is well-defined and lands on a point $x_t \in \textswab{C} \cap I_1$. Fix, $J \in \mathcal{J}(Q_1,Q_2)$. Then there is some $Q=I \times J$ that belongs to $\mathcal{Q}(Q_1,Q_2)$ and such that $x_t \in I$.  There is a unique curve $\gamma_{t_J}  \in \Gamma(Q)$ that joins $(x_t,\sup J)$ to $(x_t,\inf J)$. Notice that all these $\gamma_{t_J} $ pass through the same Cantor point on each of the horizontal edges of each $Q$.

Define a function $ \eta_t\colon \bigcup \mathcal J(Q_1,Q_2) \to \mathbf X$ by $\eta_t|_J = \gamma_{t_J} $, where $\gamma_{t_J} $ was identified in the preceding paragraph. By a triangle inequality, we can show that $\eta_t$ is Lipschitz, with the Lipschitz constant the same as the common Lipschitz constant that we have had for all families $\Gamma(I)$ and $\Gamma_Q$ thus far.

Since $\bigcup \mathcal J(Q_1,Q_2)$ is dense in $[\sup J_2,\inf J_1]$, the function $\eta_t$ has a unique Lipschitz extension $\eta_t\colon [\sup J_2,\inf J_1] \to \mathbf X$ with the same Lipschitz constant. We denote the collection of all $\eta_t, t \in I_1$ by $\Gamma(Q_1,Q_2)$. Analogous path family is built when $Q_2$ is the first large cube \emph{above} $Q_1$.

\section{Polar Integration Along the Path Families in \protect\boldmath {$\mathbf X$}}
\label{sec:polar-int}
\subsection{Integration along the path family \protect\boldmath {$\Gamma(I)$}}
Recall that $I_{01}=[0,1]$, and $Q_{I_{01}}=[0,1]\times[0,\lambda^*]$. We decompose the trace of $\Gamma(I)$ into pieces $E_{ij}$ according to Figure~\ref{curtains}.

\begin{figure}[h]
    \centering
    \begin{tikzpicture}[line cap=round,line join=round,scale=0.25]

  \draw[thick] (0,0) rectangle (32,16);
    \draw[dashed] (0,4) -- (8,4);
    \draw[dashed] (0,2) -- (8,2);    
    \draw[dashed] (0,1) -- (8,1);
  \foreach \x in {12,12.5,13,13.5,14,14.5,15,15.5,16}
  {
    \draw (0.5*\x,4) -- (\x,16);
    }

\foreach \x in {0,1,2,3,4,5,6,7,8}
{
        \draw (8-0.125*\x,2) -- (8-0.25*\x,4);
}

\foreach \x in {0,2,4,6,8,10,12,14,16}
{
    \draw (8-0.03125*\x,1) -- (8-0.0625*\x,2);
}

\node at (-1.5,10) {$\lambda/2$};
\node at (17,13.5) {$L_{ij}$};
  \coordinate (A) at (16.5,14.4);
  \coordinate (B) at (15,15.8);
  \draw[->, thick] (A) to (B);

    \node at (10,1.5) {$E_{ij}$};
    \coordinate (C) at (9,1.5);
  \coordinate (D) at (8.2,1.5);
  \draw[->, thick] (C) to (D);

      \node at (11,4.5) {$L'_{ij}$};
    \coordinate (E) at (10,4);
  \coordinate (F) at (7.7,2.1);
  \draw[->, thin] (E) to (F);
\end{tikzpicture}
    \caption{Decomposition of the trace of $\Gamma(I)$.}
    \label{curtains}
\end{figure}
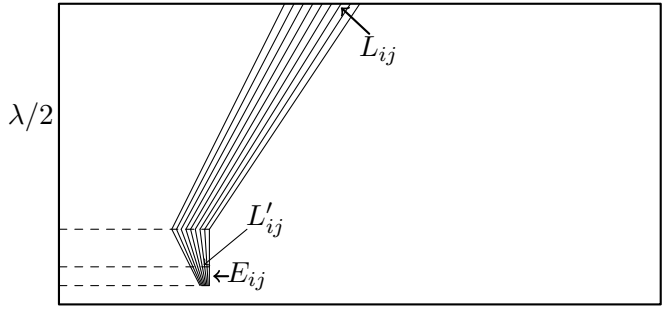

Namely, $i\in\mathbb N$ counts the layers and $j=1,\ldots,2^i$ indexes the congruent pieces within layer $i$. One can write explicit description for the sets $E_{ij}$, but there is not much gain in doing so.

Each $E_{ij}$ is a trapezoid with the ratio of the lengths of its longer to shorter side depending only on $\lambda$. We denote by $L'_{ij}$ the top horizontal edge of $E_{ij}$. There is a subset $L_{ij} \subset [0,1]$ such that $E_{ij}$ intersects and is covered by the trace of the paths $\gamma_t \in \Gamma(I_{01}), t \in L_{ij}$. Notice that $\cup_j L_{ij} = [0,1]$.

Clearly,
\begin{equation}\label{Lij-L'}
    |L_{ij}| = 2^{-i}, \quad |L'_{ij}| = \lambda^{i-1}/2.
\end{equation}

Let $b\colon L'_{ij} \to L_{ij}$ be the canonical bijection. Then by the coarea formula (Lemma~\ref{lem:coarea-intgr}), there exists a constant $C>0$, which depends only on $\lambda$, such that for every measurable function $g\colon Q_{I_{01}} \to [0,\infty]$ 
\begin{equation}
\label{fubp2}
\int_{L'_{ij}}\Bigl(\int_{\gamma_{b(t)}} g \chi_{E_{i,j}} \, ds\Bigr) \, dt \leq C \int_{E_{i,j}} g \, d\mu.
\end{equation}
From \eqref{Lij-L'} and \eqref{fubp2}, and a simple $1$-dimensional change of variables via the bijection $b$, we deduce that:
\begin{equation*}
    \int_{L_{i,j}}\int_{\gamma_t}g\chi_{E_{i,j}}dsdt \le C (2\lambda)^{-i}\int_{E_{i,j}} g \, d\mu,
\end{equation*}
where $C$ again depends only on $\lambda$.

By summation over $j$ we conclude the following (where $E_i:=\cup_j E_{ij}$):
\begin{lemma}
    \label{Ei-fubini}
    There is a 
    constant $C>0$, only depending on $\lambda$, such that for every measurable function $g: Q_{I_{01}}\to[0,\infty]$ and every $i\in\mathbb N $,
    \begin{equation}
        \label{ei-fub-eq}
    \int_{I_{01}}\int_{\gamma_t} g \chi_{E_i} \, dsdt \le C {(2\lambda)^{-i}}\int_{E_i} g \, d\mu.
    \end{equation}
\end{lemma}

\begin{proposition}
\label{kernel-gamI0}
Let $p > 2-{\rm dim} \ \textswab{C}$. There is a constant $C>0$, depending only on $p$ and $\lambda$, such that for every measurable function $g:Q_{I_{01}}\to[0,\infty]$,
\begin{equation*}
        \int_{I_{01}}\int_{\gamma_t} g \, dsdt \leq C \Big(\int_{Q_{I_{01}}} g^p \, d\mu\Big)^{1/p}.
\end{equation*}
\end{proposition}
\begin{proof}
    We apply H\"older's inequality to the integral on the right-hand side of \eqref{ei-fub-eq}, then sum over $i$, and apply a H\"older's inequality to the discrete sum, and obtain
\begin{align*}
        \int_{I_{01}}\int_{\gamma_t} g \, dsdt &= \sum_{i=0}^\infty \int_{I_{01}} \int_{\gamma_t} g \chi_{E_i} \, dsdt \notag \\
        & \lesssim \sum_{i=0}^\infty {(2\lambda)^{-i}}\int_{E_i} g \, d\mu \notag \\
        & \lesssim \sum_{i=0}^\infty (2\lambda)^{-i}\Bigl(\mu(E_i)\Bigr)^{(p-1)/p} \Bigl(\int_{E_i} g^p \, d\mu\Bigr)^{1/p}\notag\\
        &\lesssim \Bigl(\sum_{i=0}^\infty (2\lambda)^{-\frac{p}{p-1}i} \mu(E_i)\Bigr)^{\frac{p-1}{p}} \Bigl(\sum_{i=0}^\infty \int_{E_i}g^pd\mu\Bigr)^{\frac{1}{p}}.
\end{align*}
Observe that $ \mu(E_i) $ is proportional, by a multiplicative factor that depends only on $\lambda$, to $2^i\lambda^{2i}$. Hence, if  $p > 2+\log 2/\log \lambda = 2-{\rm dim} \ \textswab{C}$, then the sum
$$
\sum_{i=0}^\infty (2\lambda)^{-\frac{p}{p-1}i} \mu(E_i)
$$
is convergent and its value depends only on $\lambda$ and $p$. The proof is complete.
\end{proof}
We write $Q_I:= I \times [0,\lambda^*|I|]$. By using the self-similarity in our constructions and following the various scaling factors, we obtain estimates for arbitrary $\Gamma(I)$:
\begin{proposition}
\label{kern-gamI-thm}
     Let $p > 2-{\rm dim} \ \textswab{C}$.
      Then there exists a constant $C>0$,
     depending only on $p$ and $\lambda$, such that for every 
     $I\in\mathcal I$
     and for every measurable function $g:Q_{I}\to[0,\infty]$,
     \begin{equation}
     \label{kern-gamI-eq}
        \intavg_{I} \int_{\gamma_t} g  \, ds dt \le  C |I|^{1-\frac{2}{p}}\Big(\int_{Q_I} g^p \, d\mu\Big)^{1/p}.
    \end{equation}
\end{proposition}
\begin{proof}
    Let $\beta \colon Q_{I_{01}} \to Q_{I}$ be the affine bijection that coincides with the canonical bijection $b \colon I_{01} \to I$ along their top edges. Then a change of variables for the line integral yields,
    $$
    \int_{\gamma_t} g  ds = |I| \int_{\gamma_{b^{-1}(t)}} (g \circ \beta) \ ds, \quad \text{for every $t \in I$}.
    $$
    Then, integrating over all $t$ and setting $t'=b^{-1}(t)$ gives
    \begin{align*}
         \intavg_{I} \int_{\gamma_t} g  dsdt=& |I|\int_{I_{01}} \Bigl(\int_{\gamma_{t'}} g \circ \beta \ ds \Bigr)dt'.
    \end{align*}
Now, Lemma~\ref{kernel-gamI0} and a change of variables for the area integral complete the proof:
    \begin{align*}
        \intavg_{I} \int_{\gamma_t} g  dsdt \lesssim |I| \Bigl( \int_{Q_{I_{01}}}g\circ \beta \ d\mu  \Bigr)^{\frac{1}{p}} & = |I| \Bigl( |I|^{-2}\int_{Q_{I}}g\ d\mu  \Bigr)^{\frac{1}{p}}. 
    \end{align*}
\end{proof}
\begin{remark}\label{rem:free-par} 
We prefer to work with integral average over $I$ as we view it as a probability measure on the path families. This will allow us to easily re-index the path families from any different interval. If we wish to not emphasize the specific interval of indexing of a path family $\Gamma$, we then write $\intavg_\Gamma F(\gamma) \, d\gamma$, where $F(\gamma)$ is a nonnegative function(al).
\end{remark}

\subsubsection{Integration along the path family \protect\boldmath {$\Gamma(Q)$}}
Recall that $\Gamma(Q)$ consists of the concatenation of two copies of $\Gamma(I)$ and a collection of vertical segments between them, see Subsection \ref{sec-pathQ}. From Proposition~\ref{kern-gamI-thm}, one proves the following:
\begin{proposition}
    \label{kern-gam-Q}
Let $p > 2-{\rm dim} \ \textswab{C}$. There exists a constant $C>0$, that depends only on $p$ and $\lambda$, such that for every $Q=I\times J$ and every measurable function $g:Q\to[0,\infty]$,
     \begin{equation}
     \label{kern-gam-Q-eq}
        \intavg_{I} \int_{\gamma_t} g  \, ds dt \leq  C|I|^{1-\frac{2}{p}}\Big(\int_{Q} g^p \, d\mu\Big)^{1/p}.
    \end{equation}
\end{proposition}

\subsection{Integration along the path family \protect\boldmath {$\Gamma({Q_1,Q_2})$}}
Very detailed and explicit calculations will be required in the rest of the paper. Hence, from now on, we limit ourselves to the case where $\textswab{D}$ itself is a self-similar Cantor set. 

Fix $\nu \in \{2,3,\ldots\}$. In the rest of this paper, we take $\textswab{D}=\frac{1}{1-2\lambda^\nu}\textswab{C}({\lambda^\nu})$. This particular scaling guarantees that the largest $J$ in $\mathcal J$ has length equal to $1$, a condition needed for the path-connectivity of $\mathbf X$.

All the cubes in $\mathbf{X}$ will become honest squares. Indeed, we have from Remark~\ref{minit1} that
\begin{equation}
    \label{divides-nu}
\forall \, J \in \mathcal{J} \; \exists \, k \in \mbbn \cup \{0\}: |J| = \lambda^{k\nu}.
\end{equation}
So, every cube $Q=I\times J$ in the construction of $\mathbf X$ will satisfy $|I|=|J|=\lambda^{k\nu}$ for some $k \in \{0,1,\ldots\}$. Observe also that $\dim \textswab{D} = \frac{1}{\nu}\dim \textswab{C}$. Recall also (Lemma~\ref{prop3.16}) that $\mathcal J$ is both dense and separated.

Suppose that $Q_2=I_2\times J_2$ is the first large cube below $Q_1=I_1\times J_1$ as defined in Subsection~\ref{secjo1}; see also Figure~\ref{path-Q1-Q2}. For every  $ Q=I\times J$ in $\mathcal{Q}(Q_1,Q_2)$, let $j$ be the integer determined by $|J|=|J_1|\lambda^j$ (equivalently, by $|I|=|I_1|\lambda^j$), and set
$$
K_p(Q_1,Q_2):=\Bigl(\sum_{Q \in \mathcal{Q}(Q_1,Q_2)} 2^{{\frac{-pj}{p-1}}}|I|^{\frac{p-2}{p-1}}\Bigr)^{\frac{p-1}{p}}.
$$
The usefulness of this constant will be apparent momentarily. But first let us prove an upper bound on 
$K_p(Q_1,Q_2)$.

For fixed $J \in \mathcal{J}(Q_1,Q_2)$, there are exactly $2^j$-many $Q$ that belong to $\mathcal{Q}(Q_1,Q_2)$ and lie on row $J$. Hence,
\begin{equation}
\label{thurpeq3}
K_p(Q_1,Q_2)^{\frac{p}{p-1}}= \sum_{J \in \mathcal{J}(Q_1,Q_2)} 2^{\frac{-j}{p-1}}|J|^{\frac{p-2}{p-1}}.
\end{equation}

Throughout this paper, we set
\begin{equation}\label{eq:p0-def}
        \textswab{p}_0:=\frac{1+\nu+2\nu\log_2\lambda}{1+\nu\log_2\lambda}.
\end{equation}
The constant $\textswab{p}_0$ will be the sharp threshold for the validity of a Poincar\'e inequality on $\mathbf X$. Remark~\ref{thresh} summarizes its asymptotic behavior.
\begin{lemma}
\label{kp-lim}
For every  $p>\textswab{p}_0$, there exists a constant 
$C>0$,
depending only on $p$, $\lambda$ and $\nu$, such that for every path family $\Gamma(Q_1,Q_2)$, with $Q_1=I_1\times J_1$,
    $$
    K_p(Q_1,Q_2) \leq C|I_1|^{1-\frac{2}{p}}.
    $$
\end{lemma}
\begin{remark}
\label{height-vs-J1}
    Notice that the bound depends on $|I_1|$, and apparently and strangely not on the ``height'' of $\Gamma(Q_1,Q_2)$. However, from the density of $\mathcal J$, the two quantities are comparable with a constant that again depends only on $\lambda$ and $\nu$.
\end{remark}
\begin{proof}[Proof of Lemma~\ref{kp-lim}]
From \eqref{thurpeq3}, and using $|J|=|I_1|\lambda^j$, 
    \begin{equation*}
\label{thurpeq2}
K_p(Q_1,Q_2)^{\frac{p}{p-1}} = \sum_{J \in \mathcal{J}(Q_1,Q_2)} |I_1|^{\frac{p-2}{p-1}}(2^{\frac{-1}{p-1}}\lambda^{\frac{p-2}{p-1}})^j.
\end{equation*}
It follows from \eqref{divides-nu} that $j=k\nu$ for some $k\in \mathbb N$. On the other hand, from the construction of $\textswab{D}$, for any $k$, {there are fewer than $2^k$-many $J \in \mathcal{J}(Q_1,Q_2)$ with $|J|=|I_1|\lambda^{k\nu}$}. Therefore, we obtain
\begin{equation}
\label{fridhap1}
K_p(Q_1,Q_2)^{\frac{p}{p-1}} \leq |I_1|^{\frac{p-2}{p-1}}\sum_{k=1}^\infty  2^k \Bigl(2^{\frac{-1}{p-1}}\lambda^{\frac{p-2}{p-1}}\Bigr)^{k\nu}.
\end{equation}
The latter sum converges if
$$
2\Big(2^{\frac{-1}{p-1}}\lambda^{\frac{p-2}{p-1}}\Big)^{\nu} <1, $$
which turns out to be equivalent to $p> \textswab{p}_0$. 
So, fix $p> \textswab{p}_0$ and let $C$ be the constant given by
$$
C^\frac{p}{p-1}:=\sum_{k=1}^\infty  2^k \Big(2^{\frac{-1}{p-1}}\lambda^{\frac{p-2}{p-1}}\Big)^{k\nu},
$$
which depends only on $p$, $\lambda$ and $\nu$. Now raising both sides of \eqref{fridhap1} to the power $\frac{p-1}{p}$ completes the proof.
\end{proof}
Since $\textswab{p}_0>2-{\rm dim} \ \textswab{C}$, the integral estimates of the previous sections for $\Gamma(I)$ and $\Gamma(Q)$ are valid for all $p>\textswab{p}_0$.
\begin{proposition}\label{kern-eta-chim}
Let $p > \textswab{p}_0$. Suppose $Q_2$ is the first large cube below/above $Q_1=I_1\times J_1$. Let $E$ be the union of all $Q\in \mathcal{Q}(Q_1,Q_2)$. Then there is a constant  
$C>0$,
depending only on $p$, $\lambda$ and $\nu$ such that for every measurable function  $g:E\to[0,\infty]$,
     \begin{equation*}
       \intavg_{I_1} \int_{\eta_t} g \chi_E  \, ds dt \le C|I_1|^{1-\frac{2}{p}}\Bigl(\int_{E} g^p d\mu\Bigr)^{1/p}.
    \end{equation*}
Here, $\eta_t, t \in I_1$ are the paths in $\Gamma(Q_1,Q_2)$.
\end{proposition}
\begin{proof}
 For every $Q=I\times J \in \mathcal{Q}(Q_1,Q_2)$, by \eqref{divides-nu}, there exists a $j\in\mathbb N$ such that $|I|=|I_1|\lambda^{j}$. The set of values $t \in I_1$ for which $\eta_t$ enters $Q$ has Lebesgue measure $2^{-j}|I_1|$. For each such $t$, the subpath of $\eta_t$ that lies in $Q$, coincides with $\gamma_{t'} \in \Gamma(Q)$ for a unique $t' \in I$, determined by a canonical bijection (see Lemma~\ref{gammacopy}).

By Proposition~\ref{kern-gam-Q}, we have
\begin{equation*}
    \intavg_{I_1}\int_{\eta_t} g\chi_Q\, dsdt = 2^{-j}\intavg_{I}\int_{\gamma_{t'}} g\chi_Q\, dsdt' \lesssim 2^{-j}|I|^{1-\frac{2}{p}}\Bigl(\int_Q g^p\, d\mu\Bigr)^{\frac{1}{p}}.
\end{equation*}
Here the first equality is obtained by a simple $1$-dimensional change of variables under the mentioned canonical bijection.

Summing the last inequalities over all $Q \in \mathcal{Q}({Q_1,Q_2})$ and applying H\"older's inequality to the discrete sum yields
\begin{align*}
\intavg_{I_1}\int_{\eta_t} g\, dsdt &= \sum_{Q \in \mathcal{Q}({Q_1,Q_2})} \intavg_{I_1}\int_{\eta_t} g\chi_Q\, dsdt \\
& \lesssim \sum_{Q \in \mathcal{Q}({Q_1,Q_2})} 2^{-j}|I|^{1-\frac{2}{p}}\Bigl(\int_Q g^p\, d\mu\Bigr)^{\frac{1}{p}} \\
&\lesssim\Bigl(\sum_{Q \in \mathcal{Q}({Q_1,Q_2})} 2^{\frac{-jp}{p-1}}|I|^{\frac{p-2}{p-1}} \Bigr)^{\frac{p-1}{p}} \Bigl(\sum_{Q \in \mathcal{Q}({Q_1,Q_2})} \int_Q g^p\, d\mu\Bigr)^{\frac{1}{p}} \\
& \lesssim K_p(Q_1,Q_2) \Bigl(\int_E g^p\, d\mu\Bigr)^{\frac{1}{p}} \\
& \lesssim |I_1|^{1-\frac{2}{p}}\Bigl(\int_E g^p\, d\mu\Bigr)^{\frac{1}{p}},
\end{align*}
where in the last step we used Lemma \ref{kp-lim}.
\end{proof}
By applying the preceding proposition to admissible functions in the calculation of $p$-modulus, we obtain:
\begin{corollary}
    \label{pos-mod}
Let $p > \textswab{p}_0$. Suppose $Q_2$ is the first large cube below $Q_1=I_1\times J_1$. Then the $p$-modulus of the path family $\Gamma(Q_1,Q_2)$ is bounded from below by $C|I_1|^{2-p}$, where $C>0$ depends only  on $p$ and $\lambda$ and $\nu$. In particular, it is positive.
\end{corollary}
\begin{proof}
Indeed, for every admissible function $g$ for $\Gamma(Q_1,Q_2)$, by definitions we have
    \begin{equation*}
       1 \le \intavg_{I_1} \int_{\eta_t} g \, ds dt.
    \end{equation*}
Hence, by Proposition~\ref{kern-eta-chim},
$$
1 \lesssim |I_1|^{1-\frac{2}{p}}\Bigl(\int_E g^p\, d\mu\Bigr)^{\frac{1}{p}},
$$
which then yields
$$
|I_1|^{2-p} \lesssim \int_E g^p\, d\mu.
$$
\end{proof}

\section{Poincar\'e Inequality on $\mathbf{X}$}\label{secpiyes}
The main motivation for the construction of the path families in the previous sections is the following main result of this paper.
\begin{theorem}
\label{main1}
Fix $\lambda \in (0,1/2)$, $\nu \in \{2,3,\ldots\}$, and let $\mathbf X$ be the space that corresponds to the choices of $\textswab{C}=\textswab{C}(\lambda)$ and $\textswab{D}=\frac{1}{1-2\lambda^\nu}\textswab{C}(\lambda^\nu)$. Then $\mathbf X$ supports a $p$-Poincar\' e inequality if and only if $p>\textswab{p}_0$, where
$$
\textswab{p}_0=\frac{1+\nu+2\nu\log_2\lambda}{1+\nu\log_2\lambda}.
$$
\end{theorem}
Theorem~\ref{main1} is proved in Sections~\ref{sec:yes-PI} and \ref{secpino}.
\begin{remark}
\label{thresh}
    The threshold $\textswab{p}_0$ satisfies $\textswab{p}_0 > 2-\dim \textswab{C}$ and has the following asymptotic behavior:
    \begin{itemize}
    \item $\textswab{p}_0\to 2-\dim \textswab{C} $ if $\nu \to \infty$, uniformly in $\lambda$,
    \item $\textswab{p}_0\to 1$ if $\lambda \to 1/2$, and
    \item $\textswab{p}_0\to 2$ if $\lambda \to 0$.
\end{itemize}
\end{remark}
This asymptotic behavior is in line with the expectation that we should get a better Poincar\'e inequality (i.e.~smaller $\textswab{p}_0$) if $\textswab{C}$ has a larger Hausdorff dimension or $\textswab{D}$ has a smaller Hausdorff dimension, see Figure \ref{fig:p0-graphs}.

\begin{figure}[h]
    \centering
\begin{tikzpicture}[scale=0.8]
\pgfmathsetmacro\vvalue{3.6}
\pgfmathsetmacro\Vlarge{2000}
\begin{axis}[
    at={(4.5cm,0)}, anchor=south west,
    width=10cm, height=6cm,
    domain=0.01:0.5, samples=200,
    xlabel={$\lambda$}, ylabel={$\textswab{p}_0$},
    xmin=0, xmax=0.5,   
    ymin=0, ymax=2.1,   
    legend pos=south west,
    every axis plot post/.style={thick},
    grid=both,
    major grid style={gray!30},
    minor grid style={gray!10},
  ]
\addplot[
  thick,
  solid,
  mark=square,            
  mark options={draw=black,fill=white}, 
  mark size=1pt,          
  mark repeat=25
]
{(1 + 2 + 4*log2(x)) / (1 + 2*log2(x))};
	   \addlegendentry{$\nu=2$}
    \addplot[black] 
      {(1 + \vvalue + 2*\vvalue*log2(x)) /
       (1 + \vvalue*log2(x))};
	 \addlegendentry{generic $\nu$}

    \addplot[blue,dashed] 
      {(1 + \Vlarge + 2*\Vlarge*log2(x)) /
       (1 + \Vlarge*log2(x))};
	   \addlegendentry{$\nu=\infty$}

    \addplot[red,   mark=*, only marks, mark size=1pt, forget plot] coordinates {(0,2)};
    \addplot[black, mark=*, only marks, mark size=1pt, forget plot] coordinates {(0,2)};
    \addplot[blue,  mark=*, only marks, mark size=1pt, forget plot] coordinates {(0,2)};
  \end{axis}
\end{tikzpicture}
   \caption{The asymptotic behavior of the threshold $\textswab{p}_0$.}
    \label{fig:p0-graphs}
\end{figure}

The fact that we can choose $\textswab{p}_0$,  $2-\dim (\textswab{C}\times \textswab{D})$, and $2-\dim \textswab{C} $ arbitrarily close to each other, by choosing $\nu$ large, will become important in the proof of Theorem \ref{main2}.

We will prove Poincar\'e inequality by establishing the pointwise inequalities of Lemma~\ref{lem:pointwise-then-PI}. In order to establish the pointwise estimates we need to find a thick family of paths between arbitrary pairs of points. To do so, we will concatenate a finite number of path families of the types $\Gamma(I), \Gamma(Q)$, and $\Gamma(Q_1,Q_2)$ from the previous sections. The key to the success of this approach is that the path families $\Gamma(Q_1,Q_2)$ contain all the small scale analysis, thus reducing the problem to the combinatorics of the finitely many large scale cubes. The integral estimates from Section~\ref{sec:polar-int} will be crucial.

Throughout this section, we will assume $2 \ge p>\textswab{p}_0$.

\subsection{The cone \protect\boldmath $\Gamma(Q_0;Q_m)$}
Let us begin with a few easy properties of the collection of intervals $\mathcal J$. Recall that for every $Q=I\times J$, $|I|=|J|$.
\begin{remark}\label{rem:properties-J}
By the self-similarity of $\textswab{D}=\frac{1}{1-2\lambda^\nu}\textswab{C}(\lambda^\nu)$, the following hold:
\begin{enumerate}
    \item By~\eqref{divides-nu}, for every $J \in \mathcal J$, $|J|=\lambda^{k\nu}$ for some $k\in \{0,1,\ldots\}$. \vspace{0.2cm}
    \item Between any distinct pair $J_1$ and $J_2$ with $|J_1|=|J_2|$ there exists a $J' \in \mathcal J$ such that $|J'| > |J_1|$ (due to Remark~\ref{minit1} and $\lambda^{-\nu}>4$). \vspace{0.2cm}
    \item If $Q_2=I_2\times J_2$ is the first large cube below (/above) $Q_1=I_1\times J_1$, then $ |I_1|< |I_2|$. Indeed, by item~(1), $ |I_1| \le \lambda^{\nu}|I_2|$.  \vspace{0.2cm}
    \item Under the assumptions in item~(3), $\dist(J_1,J_2) \approx |J_1|$, with constants that depend only on $\lambda$ and $\nu$. This follows from the density and separatedness of $\mathcal J$ as in Definition~\ref{uniform-separattion} and Definition~\ref{def:density}.
\end{enumerate}    
\end{remark}

We now construct path families that connect cubes of different sizes that are vertically aligned (Figure~\ref{cone-Qx}).\footnote{Here, and later $M$ does not stand for an index. Rather $Q_M$ is thought of as ``the middle'' cube.}
\begin{lemma}
    \label{pnt-2}
Suppose that cubes $Q_0=I_0 \times J_0$ and $Q_M=I' \times J'$ satisfy $\sup J' < \inf J_0$, $I_0 \subset I'$ and for every $J$ between $J_0$ and $J'$ we have $|J|<|J'|$. Then, there exists $m \ge 1$ and a sequence of cubes $Q_1,\ldots, Q_{m-1}, Q_m$, such that $Q_m=Q_M$, and for each $1\leq i\leq m$, cube $Q_i$ is the first large cube below $Q_{i-1}$. Moreover, for every $1\le i\leq m$, we have 
\begin{equation}\label{sum-Ii}
    \dist (Q_0,Q_i)  \approx \sum_{j=0}^{i-1} |I_j| \approx |I_{i-1}|,
\end{equation}
where the comparison constants depend only on $\lambda$ and $\nu$.
\end{lemma}
\begin{remark}
    Analogous statement holds when $Q_M$ is \emph{above} $Q_0$.
\end{remark}
\begin{proof}
In light of the properties listed in Remark~\ref{rem:properties-J}, the only slightly non-trivial claim is the second comparison in \eqref{sum-Ii}. But this follows again from Remark~\ref{rem:properties-J} which shows that $\sum_{j=0}^{i-1} |I_j|$ is dominated by a geometric sum with ratio $\lambda^{\nu}$ and the largest term being $|I_{i-1}|$. Thus, the two quantities are comparable.
\end{proof}

\begin{figure}
   \centering
\begin{tikzpicture}[scale=0.15]
  \draw[thick] (-2,118) -- ++(68,0);

  \fill[gray!30] (-2,110) rectangle ++(68,8);

    \foreach \k in {126} 
    \foreach \l in {0,48} \draw[line width=.2pt,fill=gray!30] (\l,\k) rectangle ++(16,16);

    \foreach \m in { 120, 144 } 
    \foreach \n in {0, 12, 48, 60} \draw[line width=.2pt,fill=gray!30] (\n,\m) rectangle ++(4,4);

    \foreach \m in {118.5, 124.5, 142.5, 148.5 } 
    \foreach \n in {0, 3, 12, 15, 48, 51, 60, 63} \draw[line width=.2pt,fill=gray!30] (\n,\m) rectangle ++(1,1);

  \foreach \x in {0,1,2,3,4,5,6,7,8,9,10,11,12,13,14,15,16}
  {
    \draw (0.5*\x,132) -- (0.25*\x,128);    
    \draw (16-0.5*\x,132) -- (16-0.25*\x,128);
}

\foreach \x in {0,0.5,1,1.5,2,2.5,3,3.5,4,4.5,5,5.5,6,6.5,7,7.5,8,8.5,9,9.5,10,10.5,11,11.5,12,12.5,13,13.5,14,14.5,15,15.5,16}
{
\draw[thin] (\x,132) -- (0.25*\x,140);
}

\foreach \x in {0,1,2,3,4,5,6,7,8}
  {
        \draw[very thin] (0.125*\x,148) -- (0.125*\x,147.5);

}

\foreach \x in {0,2,4,6,8}
 {
  \draw [thin](0.125*\x,124.5) -- (0.1255*\x,125.5);
        \draw [thin](4-0.125*\x,124.5) -- (4-0.1255*\x,125.5);
        \draw [thin](12+0.125*\x,124.5) -- (12+0.1255*\x,125.5);
        \draw [thin](15+0.125*\x,124.5) -- (15+0.1255*\x,125.5);

    \draw [thin](0.125*\x,118.5) -- (0.1255*\x,119.5);
        \draw [thin](16-0.125*\x,118.5) -- (16-0.1255*\x,119.5);
        \draw [thin](12+0.125*\x,118.5) -- (12+0.1255*\x,119.5);
        \draw [thin](3+0.125*\x,118.5) -- (3+0.1255*\x,119.5);

     \draw [thin](0.125*\x,142.5) -- (0.1255*\x,143.5);
        \draw [thin](3+0.125*\x,142.5) -- (3+0.1255*\x,143.5);

}

\foreach \x in {0,1,2,3,4,5,6,7,8}
{
 \draw [very thin](0.125*\x,148.5) -- (0.1255*\x,149);
}

\foreach \x in {0,1,2,3,4,5,6,7,8}
  {
    \draw[thin] (4-0.125*\x,120) -- (4-0.25*\x,122);
        \draw[thin] (0.125*\x,120) -- (0.25*\x,122);
        \draw[thin] (0.25*\x,122) -- (0.125*\x,124);
        \draw[thin] (4-0.25*\x,122) -- (4-0.125*\x,124);

        \draw[thin] (16-0.125*\x,120) -- (16-0.25*\x,122);
        \draw[thin] (12+0.125*\x,120) -- (12+0.25*\x,122);
        \draw[thin] (12+0.25*\x,122) -- (12+0.125*\x,124);
        \draw[thin] (16-0.25*\x,122) -- (16-0.125*\x,124);

        \draw[thin] (0.25*\x,140) -- (0.125*\x,142);
        \draw[thin] (4-0.25*\x,140) -- (4-0.125*\x,142);
  
  \draw[thin] (4-0.125*\x,126) -- (4-0.25*\x,128);  
        \draw[thin] (0.125*\x,126) -- (0.25*\x,128);
        \draw[thin] (12+0.125*\x,126) -- (12+0.25*\x,128);
        \draw[thin] (16-0.125*\x,126) -- (16-0.25*\x,128);
}

\foreach \x in {0,1,2,3,4,5,6,7,8}
{
        \draw[thin] (0.125*\x,144) -- (0.25*\x,146);
        \draw[thin] (4-0.125*\x,144) -- (4-0.25*\x,146);
}

\foreach \x in {0,0.5,1,1.5,2,2.5,3,3.5,4,4.5,5,5.5,6,6.5,7,7.5,8}
{
 \draw[very thin] (0.5*\x,146) -- (0.125*\x,147.5);
}

\draw [decorate, decoration={brace, mirror}]
      (18,118.2) -- (18,125.9)
      node [black, midway, right=1pt] {\scalebox{0.8}{$\Gamma(Q_2,Q_3)$}};

\draw [decorate, decoration={brace, mirror}]
      (17,126.1) -- (17,131.9)
      node [black, midway, right=1pt] {\scalebox{0.8}{$\Gamma(I_2)$}};

\draw [decorate, decoration={brace, mirror}]
      (15.5,133) -- (5,139.9)
      node [black, midway, right=1pt] {\scalebox{0.8}{$\Delta_2$}};

\node at (6,141) {\scalebox{0.5}{$\Gamma(I_1)$}};      
\node at (32,115) {\scalebox{0.8}{$Q_M=Q_3$}};
\node at (-1.8,149) {\scalebox{0.8}{$Q_0$}};
\node at (5.5,146) {\scalebox{0.8}{$Q_1$}};
\node at (17.5,138) {\scalebox{0.8}{$Q_2$}};
\end{tikzpicture}
\caption{The path family $\Gamma(Q_0;Q_M)$. Clearly, some fine features of the paths within individual cubes are not present in the figure.}
    \label{cone-Qx}
\end{figure}

Let $Q_0$ and $Q_M$ satisfy the assumptions of Lemma~\ref{pnt-2} and let $Q_1,\ldots,Q_m$, with $Q_m=Q_M$, be the sequence of intermediate cubes given by its claim. We wish to build a path family from $Q_0$ to $Q_M$ (Figure~\ref{cone-Qx}).

For each $1 \le i \le m$, the path family $\Gamma(Q_{i-1},Q_{i})$ is well-defined. If $m=1$, then $\Gamma(Q_0,Q_M)$ is the desired path family. So, suppose $m\ge 2$. The path family $\Gamma(Q_0,Q_1)$ ends on the top edge of $Q_1$ and the path family $\Gamma(Q_1,Q_2)$ starts along the bottom edge of $Q_1$ (Figure~\ref{cone-Qx}).

In order to join them inside $Q_1$, we attach an upside-down copy of $\Gamma(I_0)$ from inside $Q_1$ to the ends of $\Gamma(Q_0,Q_1)$. This is possible because $|I_1|\ge \lambda^{-\nu} |I_0|$. Similarly, we attach a copy of $\Gamma(I_1)$ from inside $Q_1$ to $\Gamma(Q_1,Q_2)$. Notice that these copies of $\Gamma(I_0)$ and $\Gamma(I_1)$ end along horizontal line segments (at different heights) inside $Q_1$, with lengths equal to $|I_0|$ and $|I_1|$, respectively. So, we concatenate them by adding line segments that connect each point from one segment to the point on the other segment that is its image under the canonical bijection.

We index this trapezoidal path family as $\Delta=\{\delta_t, t\in I_0\}$. Let us denote its trace by $T$. Note that $|I_0|$, the length of the top edge of $T$, can be much smaller than $|I_1|$, the length of its bottom edge.

We now specify exactly how we concatenate $\Gamma(Q_{0},Q_{1})$ to $\Delta$ and then to $\Gamma(Q_{1},Q_{2})$. For every $t \in I_0$, with the usual exception of countably many, there is the associated path $\eta_t$ in $\Gamma(Q_{0},Q_{1})$ and the associated path $\gamma_t$ in (the upside down copy of) $\Gamma(I_0)$. These two meet on the same point along the top edge of $Q_1$, so we can concatenate them in the obvious way.

Now, let $b$ be the bijection from the top edge of $T$ to its bottom edge. We continue our paths by joining $t$ to $b(t)$ via a straight line. This line is $\delta_t \in \Delta$. Then we continue with $\gamma_{b(t)}$ in the copy of $\Gamma(I_1)$ inside $Q_1$. Finally, we add the path $\eta_{b(t)} \in \Gamma(Q_{1},Q_{2})$. As they meet on the same point along the edge of $Q_1$ this is possible.

By appropriate re-parameterization, we have now found a single path associated to (almost) every $t \in I_0$ that begins on the bottom edge of $Q_0$ and ends on the top edge of $Q_2$.

Write $\Delta_1:=\Delta$, $T_1:=T=tr(\Delta_1)$. We repeat the construction of the auxiliary trapezoidal paths inside $Q_2,\ldots,Q_{m-1}$, and obtain the families $\Delta_i$, $1\le i \le m-1$, and their traces $T_i$. We use $\Delta_i$'s and copies of $\Gamma(I_i)$ to concatenate the families $\Gamma(Q_{i-1},Q_i)$, $i=1,\ldots,m$ and form a single family, denoted\footnote{Note the semicolon as opposed to a comma.} by $\Gamma(Q_0;Q_m)$, of paths that join the bottom edge of $Q_0$ to the top edge of $Q_m=Q_M$.

To summarize, here is what $\Gamma(Q_0;Q_m)$ consists of (assume $m \ge 3$; the cases of $m=1,2$ are simple, indeed, if $m=1$, then $\Gamma(Q_0;Q_m)$ coincides with $\Gamma(Q_0,Q_m)$):
\begin{align*}
    & \Gamma(Q_0,Q_1), \text{copy of $\Gamma(I_0)$}, \Delta_1, \text{copy of $\Gamma(I_1)$}, \Gamma(Q_1,Q_2), \\
    & \text{copy of $\Gamma(I_1)$}, \Delta_2, \ldots, \\
    & \text{copy of $\Gamma(I_{m-2})$}, \Delta_{m-1}, \text{copy of $\Gamma(I_{m-1})$}, \Gamma(Q_{m-1},Q_m).
\end{align*}
Observe that there are two copies of each of $\Gamma(I_i)$ for $1\le i \le m-2$, while there is only one copy of each of $\Gamma(I_0)$ and $\Gamma(I_{m-1})$. By Remark~\ref{rem:free-par}, we can re-index, if necessary, each of the component path families over $I_0$. So, we index the paths in $\Gamma(Q_0;Q_m)$ as $\xi_t, t \in I_0$. 

\begin{proposition}\label{lem:dbl-intg-cone-full}
Fix an arbitrary $x \in Q_0$. Then for every measurable $g\colon \mathbf{X} \to [0,\infty]$,
    \begin{equation}\label{eq:cone-double}
    \intavg_{I_0}\int_{\xi_t} g\, dsdt \lesssim \dist(x,Q_M) \Bigl(\mathcal M_{C\dist(x,Q_M)}g^p(x)\Bigr)^{1/p},
\end{equation}
where $C$ and the comparison constants depend only on $\lambda$, $\nu$ and $p$.
\end{proposition}
\begin{proof}
We will assume $m\ge 2$ as the case of $m=1$ requires only the estimate \eqref{coneq33} below. By the construction and Remark~\ref{rem:properties-J}, each of the families $\Delta_i$ that appear in the construction of $\Gamma(Q_0;Q_M)$ satisfies the assumptions of Corollary~\ref{cor:coare-max-2} with respect to $x$. Hence, with $T_i=tr(\Delta_i)$,
\begin{equation}\label{coneq11}
    \intavg_{I_0}\int_{\xi_t} g \chi_{T_i}\, dsdt \lesssim |I_i| \mathcal M_{Cd(x,Q_M)} g(x) \le |I_i| (\mathcal M_{Cd(x,Q_M)}g^p(x))^{1/p},
\end{equation}
where the last claim follows from H\"older's inequality. The comparison constants, here and in what follows, depend only on $\lambda$ and $\nu$ and $p$.

Let $\Gamma(I_i)$ be any of the path families that appear in the construction of $\Gamma(Q_0;Q_M)$. Then $tr(\Gamma(I_i))$ is at a distance from $x$ comparable to $|I_i|$ and has measure comparable to $|I_i|^2$. Together with Proposition~\ref{kern-gamI-thm} these prove that
\begin{equation}\label{coneq22}
    \intavg_{I_0}\int_{\xi_t} g \chi_{tr(\Gamma_i)}\, dsdt \lesssim  |I_{i}|\Bigl( \intavg_{tr(\Gamma_i)} g^pd\mu \Bigr)^{1/p}  \lesssim |I_{i}| \Bigl(\mathcal M_{C\dist(x,Q_M)}g^p(x)\Bigr)^{1/p}.
\end{equation}
Finally, denote by $E_i$ the trace of $\Gamma(Q_{i-1},Q_i)$, $i=1,\ldots,m$. By Remark~\ref{rem:properties-J}, $E_i$ is, roughly, at the right distance from $x$ and of the right measure so that by Proposition~\ref{kern-eta-chim} we obtain 
\begin{equation}\label{coneq33}
    \intavg_{I_0}\int_{\xi_t} g \chi_{E_i}\, dsdt \lesssim |I_{i}|\Bigl(\intavg_{E_i} g^pd\mu \Bigr)^{1/p} \lesssim |I_{i}|\Bigl(\mathcal M_{C\dist(x,Q_M)}g^p(x)\Bigr)^{1/p}.    
\end{equation}
Since, the path families that comprise $\Gamma(Q_0;Q_M)$ have essentially disjoint traces, by the additivity of the line integrals and the estimates \eqref{coneq11}, \eqref{coneq22} and \eqref{coneq33} we get
\begin{align*}
    \intavg_{I_0}\int_{\xi_t} g\, dsdt \lesssim \Bigl(\sum_{i=1}^{m-1} |I_{i}|\Bigr) \Bigl(\mathcal M_{C\dist(x,Q_M)}g^p(x)\Bigr)^{1/p} .
\end{align*}

By \eqref{sum-Ii} we have $\sum_{i=1}^{m-1} |I_{i}| \approx \dist(x,Q_M)$. The proof is complete.    
\end{proof}
\subsection{The double cone \protect\boldmath {$\Gamma_{x,y}$}}
We use the cones from the previous section to build double-cones that connect arbitrary cubes in $\mathbf X$.
\begin{lemma}
    \label{QM-cube-case}
Fix any pair of distinct cubes $Q_x=I_x\times J_x$ and $Q_y=I_y\times J_y$ in $\mathbf X$ and, without loss is generality, assume $|I_y| \ge |I_x|$. Then exactly one of the following holds:
\begin{itemize}
    \item [(a)] $I_x \subset I_y$ and for every $J$ between $J_x$ and $J_y$ we have $|J| < |I_y|$;
    \item [(b)] $I_x \subset I_y$ and for some $J$ between $J_x$ and $J_y$ we have $|J| \ge |I_y|$;
    \item [(c)] $I_x \cap I_y \ne \emptyset$.
\end{itemize}
In case (a) the path family $\Gamma(Q_x;Q_y)$ is well-defined. In case (b),  there exists a cube $Q_M=I_M\times J_M$ such that both $\Gamma(Q_x;Q_M)$ and $\Gamma(Q_y;Q_M)$ are well-defined and
    \begin{equation}
        \label{eq:case2-QM}
        |J_M|+\dist(Q_x,Q_M)+\dist(Q_y,Q_M) \approx \dist(Q_x,Q_y).
    \end{equation}
In case (c), there exists a cube $Q_M=I_M\times J_M$ such that both $\Gamma(Q_x;Q_M)$ and $\Gamma(Q_y;Q_M)$ are well-defined and there are two sub-cases: (c1) $J_M$ is between $J_x$ and $J_y$ and
    \begin{equation}
    \label{eq:case1-QM}
        |J_M|+\dist(Q_x,Q_M)+\dist(Q_y,Q_M) \approx \dist(Q_x,Q_y),
    \end{equation}
or (c2) both $J_x$ and $J_y$ are on the same side of $J_M$ and
    \begin{equation}
        \label{eq:case1-QM-c2}
        \dist(I_x,I_y)+\dist(Q_x,Q_M)+\dist(Q_y,Q_M) \approx \dist(Q_x,Q_y).
    \end{equation}
\end{lemma}
\begin{remark}
The cases in \eqref{eq:case1-QM-c2} is different from the others since, in comparison, the cube $Q_M$ can be much larger than $Q_x$ and $Q_y$, so, $|J_M|$ can be much larger than $\dist(I_x,I_y)$; see Figure~\ref{fig:full-dbl-cone}. The cube $Q_M$ is not necessarily unique, e.g.\ in case (c2).
\end{remark}

\begin{proof}[Proof of Lemma~\ref{QM-cube-case}]
By Remark~\ref{rem-july-1} part (3), the cases (a), (b) and (c) are exhaustive. The claim about case (a) is the very content of the previous section. In case (b), we first identify $J_M \in \mathcal J$ as the largest $J$ between $J_x$ and $J_y$. This is unique by part (2) of Remark~\ref{rem:properties-J}. Then we let $Q_M=I_M\times J_M$ to be the unique cube that satisfies $I_M \supset I_y$. It is now easy to see that the conditions of the previous section are satisfied with $Q_x$, resp.\ $Q_y$, and $Q_M$. The estimate \eqref{eq:case2-QM} is a consequence of \eqref{sum-Ii}.

For case (c), observe that the set of cubes $Q=I\times J$ with the property that $I_y \subset I$ is finite and nonempty. We choose the ones that have the minimal $|I|$, and among them, we choose one that meets the required distance condition. Such choice always exists but may not be unique. The cones are well-defined and the claimed distance estimates follow again from \eqref{sum-Ii}.
\end{proof}

\begin{figure}[t]
     \centering
\begin{tikzpicture}[scale=0.12]
  \draw[thick] (-2,118) -- ++(68,0);

  \fill[gray!30] (-2,78) rectangle ++(68,118-78);

    \foreach \k in {126} 
    \foreach \l in {0,48} \draw[line width=.2pt,fill=gray!30] (\l,\k) rectangle ++(16,16);

    \foreach \m in { 120, 144 } 
    \foreach \n in {0, 12, 48, 60} \draw[line width=.2pt,fill=gray!30] (\n,\m) rectangle ++(4,4);

    \foreach \m in {118.5, 124.5, 142.5, 148.5 } 
    \foreach \n in {0, 3, 12, 15, 48, 51, 60, 63} \draw[line width=.2pt,fill=gray!30] (\n,\m) rectangle ++(1,1);

  \foreach \x in {0,1,2,3,4,5,6,7,8,9,10,11,12,13,14,15,16}
  {
    \draw (0.5*\x,132) -- (0.25*\x,128);    
    \draw (16-0.5*\x,132) -- (16-0.25*\x,128);
}

\foreach \x in {0,0.5,1,1.5,2,2.5,3,3.5,4,4.5,5,5.5,6,6.5,7,7.5,8,8.5,9,9.5,10,10.5,11,11.5,12,12.5,13,13.5,14,14.5,15,15.5,16}
{
\draw[thin] (\x,132) -- (0.25*\x,140);
}

\foreach \x in {0,1,2,3,4,5,6,7,8}
  {
        \draw[very thin] (0.125*\x,148) -- (0.125*\x,147.5);

}

\foreach \x in {0,2,4,6,8}
 {
  \draw [thin](0.125*\x,124.5) -- (0.1255*\x,125.5);
        \draw [thin](4-0.125*\x,124.5) -- (4-0.1255*\x,125.5);
        \draw [thin](12+0.125*\x,124.5) -- (12+0.1255*\x,125.5);
        \draw [thin](15+0.125*\x,124.5) -- (15+0.1255*\x,125.5);

    \draw [thin](0.125*\x,118.5) -- (0.1255*\x,119.5);
        \draw [thin](16-0.125*\x,118.5) -- (16-0.1255*\x,119.5);
        \draw [thin](12+0.125*\x,118.5) -- (12+0.1255*\x,119.5);
        \draw [thin](3+0.125*\x,118.5) -- (3+0.1255*\x,119.5);

    \draw [thin](48+0.125*\x,118.5) -- (48+0.1255*\x,119.5);
        \draw [thin](51+0.125*\x,118.5) -- (51+0.1255*\x,119.5);

     \draw[very thin] (0.125*\x,142.5) -- (0.1255*\x,143.5);
        \draw[very thin](3+0.125*\x,142.5) -- (3+0.1255*\x,143.5);

}

\foreach \x in {0,2,4,6,8}
{
 \draw [very thin](0.125*\x,148.5) -- (0.1255*\x,149);
}

\foreach \x in {0,1,2,3,4,5,6,7,8}
  {
    \draw[thin] (4-0.125*\x,120) -- (4-0.25*\x,122);
        \draw[thin] (0.125*\x,120) -- (0.25*\x,122);
        \draw[thin] (0.25*\x,122) -- (0.125*\x,124);
        \draw[thin] (4-0.25*\x,122) -- (4-0.125*\x,124);

        \draw[thin] (16-0.125*\x,120) -- (16-0.25*\x,122);
        \draw[thin] (12+0.125*\x,120) -- (12+0.25*\x,122);

        \draw[thin] (52-0.125*\x,120) -- (52-0.25*\x,122);
        \draw[thin] (48+0.125*\x,120) -- (48+0.25*\x,122);
        
        \draw[thin] (12+0.25*\x,122) -- (12+0.125*\x,124);
        \draw[thin] (16-0.25*\x,122) -- (16-0.125*\x,124);

        \draw[thin] (0.25*\x,140) -- (0.125*\x,142);
        \draw[thin] (4-0.25*\x,140) -- (4-0.125*\x,142);
  
  \draw[thin] (4-0.125*\x,126) -- (4-0.25*\x,128);  
        \draw[thin] (0.125*\x,126) -- (0.25*\x,128);
        \draw[thin] (12+0.125*\x,126) -- (12+0.25*\x,128);
        \draw[thin] (16-0.125*\x,126) -- (16-0.25*\x,128);
}

\foreach \x in {0,1,2,3,4,5,6,7,8}
{
        \draw[very thin] (0.125*\x,144) -- (0.25*\x,146);
        \draw[very thin] (4-0.125*\x,144) -- (4-0.25*\x,146);
}

\foreach \x in {0,0.5,1,1.5,2,2.5,3,3.5,4,4.5,5,5.5,6,6.5,7,7.5,8}
{
 \draw[very thin] (0.5*\x,146) -- (0.125*\x,147.5);
}

\foreach \x in {0,1,2,3,4,5,6,7,8}
{
 \draw[thin] (4-0.25*\x,116) -- (4-0.125*\x,118);  
\draw[thin] (0.25*\x,116) -- (0.125*\x,118);
\draw[thin] (12+0.25*\x,116) -- (12+0.125*\x,118);
 \draw[thin] (16-0.25*\x,116) -- (16-0.125*\x,118);

  \draw[thin] (48+0.25*\x,116) -- (48+0.125*\x,118);
 \draw[thin] (52-0.25*\x,116) -- (52-0.125*\x,118);  
 }

\foreach \x in {0,1,2,3,4,5,6,7,8,9,10,11,12,13,14,15,16}
  {
    \draw[very thin] (0.5*\x,112) -- (0.25*\x,116);    
    \draw[very thin] (16-0.5*\x,112) -- (16-0.25*\x,116);
}

\foreach \x in {0, 0.5, 1, 1.5, 2, 2.5, 3, 3.5, 4, 4.5, 5, 5.5, 6, 6.5, 7, 7.5, 8, 8.5, 9, 9.5, 10, 10.5, 11, 11.5, 12, 12.5, 13, 13.5, 14, 14.5, 15, 15.5, 16}
{
        \draw[very thin] (\x,112) -- (22,82+1.3*\x); 
        \draw[very thin] (22,82+1.3*\x) -- (52-0.25*\x,116); 
}
        
\node at (-1.8,149) {\scalebox{0.8}{$Q_x$}};
\node at (46,122) {\scalebox{0.8}{$Q_y$}};
\node at (31,110) {\scalebox{0.8}{$\Delta_M$}};
\end{tikzpicture}
  \caption{The family of paths from $Q_x$ to $Q_y$, joined with a family $\Delta_M$ inside $Q_M$.}
   \label{fig:full-dbl-cone}
\end{figure}
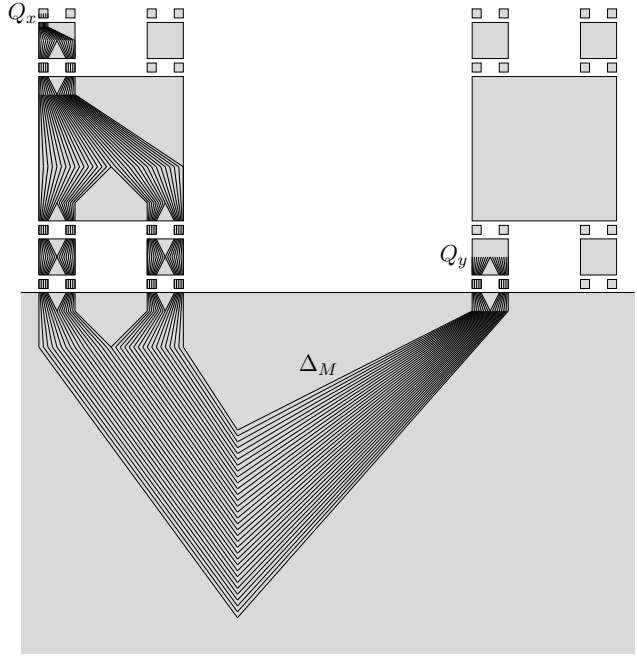

In the context of Lemma~\ref{QM-cube-case}, let $Q_x=Q_0, Q_1,\ldots, Q_{m_x-1}, Q_{m_x}=Q_{M}$, where $m_x \ge 1$, be the sequence of cubes in the definition of the cone $\Gamma(Q_x;Q_M)$ (Lemma~\ref{pnt-2}). We glue of a copy of $\Gamma(I_{m_x-1})$, from inside $Q_M$, to the end of $\Gamma(Q_x;Q_M)$. We denote by $I_{x,M}$ the horizontal line segment inside $Q_M$ along which these paths end. Observe that $|I_{x,M}|=|I_{m_x-1}|$. Unless we are in case (a) of Lemma~\ref{QM-cube-case}, we similarly find $I_{m_y-1}$ and $I_{y,M}$.

The cones $\Gamma(Q_x;Q_M)$ and $\Gamma(Q_y;Q_M)$ will do the most intricate part of connecting the cubes $Q_x$ and $Q_y$ by a path family. However, we will need to fill the gap between the two inside $Q_M$. The next lemma guarantees that we can do so while retaining desirable integral estimates.
\begin{lemma}\label{lem:Delta-M-fam}
There exists a path family $\Delta_M= \{\delta_t: t \in I_{x,M}\}$ such that $\delta_t$ joins $t \in I_{x,M}$ to $b(t) \in I_{y,M}$, where $b\colon I_{x,M} \to I_{y,M}$ is the canonical bijection, and for every measurable $g\colon \mathbf X \to [0,\infty]$ we have
    \begin{equation}\label{eq:kern-Delta-M}
    \intavg_{I_{x,M}} \int_{\delta_t} g\, dsdt \lesssim d(x,y)(\mathcal M_{Cd(x,y)} g(x) + \mathcal M_{Cd(x,y)} g(x)),
\end{equation}
with a comparison constant that depends only on $p, \nu, \lambda$.
\end{lemma}
\begin{proof}
From the constructions, we know $d(x,I_{x,M}) \approx |I_{x,M}|$ and similarly $d(y,I_{y,M}) \approx |I_{y,M}|$. Observe that $I_{x,M}$ and $I_{y,M}$ might both be near the same edge of $Q_M$ or be near the opposite sides of it. In either case, 
$$
\dist(I_{x,M},I_{y,M}) \gtrsim \max\{|I_{x,M}|,|I_{y,M}|\}.
$$
Within $Q_M$ we are practically in $\mathbb R^2$. The desired path family $\Delta_M$ can be constructed by gluing together two path families that join $I_{x,M}$, resp.\ $I_{y,M}$, to a common line segment in $Q_M$ of length comparable to $\dist(I_{x,M},I_{y,M})$, positioned closer to the larger of $I_{x,M}$ and $I_{y,M}$ (Figure~\ref{fig:full-dbl-cone}).

There is flexibility in the choices above. We omit the details, but by analysis similar to that in Section~\ref{sec:coare-trapez}, and the distance estimates above, we can guarantee \eqref{eq:kern-Delta-M}.
\end{proof}
Suppose $x$ and $y$ belong to different cubes $Q_x= I_x\times J_x$ and $Q_y=I_y\times J_y$, respectively. We wish to build a family of paths that join $x$ to $y$.

Let us consider the most involved case that will require all of the types of the path families that we have constructed so far. Namely, suppose we are in either case (b) or (c) of Lemma~\ref{QM-cube-case}. Other cases can be treated by simpler versions of the analysis below as fewer path families will be required.

Within $Q_x$ we construct a family of paths $\Delta_x=\{\delta_t, t \in I_x\}$ that start from $x$ and end along the copy of $I_x$ where $\Gamma(I_x)$ start. If $x$ is far enough from $I_x$ we can simply take the radial projection from $x$. If $x$ is near to or on the segment $I_x$, then we project radially to another horizontal segment and then from that to $I_x$. It will not be a problem if these paths self-intersect or their traces intersect the traces of $\Gamma(I_x)$. The important fact is that we can guarantee (as in in Section~\ref{sec:coare-trapez}) that for every measurable $g \colon Q_x \to [0,\infty]$
\begin{equation}\label{eq:Delta-x}
    \intavg_{\Delta_x}\int_{\gamma_t} g\, dsdt \lesssim \diam (Q_x) \mathcal M_{(\diam Q_x)}g(x).
\end{equation}
We consider the analogous path family $\Delta_y$ inside $Q_y$ such that for every measurable $g \colon Q_y \to [0,\infty]$
\begin{equation}\label{eq:Delta-y}
    \intavg_{\Delta_y}\int_{\gamma_t} g\, dsdt \lesssim \diam (Q_y) \mathcal M_{(\diam Q_y)}g(y).
\end{equation}
We concatenate, in the order below, all of the following path families (Figure~\ref{fig:full-dbl-cone}). (Remember from Remark~\ref{rem:free-par} that we can easily re-index any of the path families from $I_x$, which we do below.)
\begin{enumerate}
    \item $\Delta_x$, which start at $x$ and end on $I_x$ inside $Q_x$,
    \item $\Gamma(I_x)$,
    \item $\Gamma(Q_x;Q_M)$,
    \item copy of $\Gamma(I_{m_x-1})$ inside $Q_M$,
    \item the family $\Delta_M$ of Lemma~\ref{lem:Delta-M-fam},
    \item copy of $\Gamma(I_{m_y-1})$, with orientation such that they end on the edge of $Q_M$,
    \item $\Gamma(Q_y;Q_M)$, oriented such that they end on the edge of $Q_y$,
    \item $\Gamma(I_y)$, with orientation such that they end on $I_y$, and
    \item $\Delta_y$, with orientation such that they end at $y$.
\end{enumerate}
As already mentioned, if we are in case (a) of Lemma~\ref{QM-cube-case}, or if $Q_x$ and $Q_y$ are the same cubes, then a path family from $x$ to $y$ can be constructed by concatenating only the appropriate subset of the path families above, and possibly a suitable replacement for $\Delta_M$. In the interest of brevity we leave such technicalities out.

Denote the path family from $x$ to $y$ by $\Gamma_{x,y}$. We are ready to prove the Poincar\'e inequality on $\mathbf X$.
\subsection{Sufficiency of \protect\boldmath {$p>\textswab{p}_0$}}\label{sec:yes-PI}
Since $\mathbf X \setminus \mathbf X'$ has zero $\mu$-measure, by Lemma~\ref{lem:pointwise-then-PI}, it suffices to prove that there are constants $C\geq 1$ and $C'>0$, such that for every (continuous) $u$ and every upper-gradient $g$ of $u$, we have
\begin{equation}
\label{pntws}
    |u(y)-u(x)| \leq C' d(x,y) \Bigl(\mathcal M_{Cd(x,y)}g^p(x))^{1/p} + \mathcal M_{Cd(x,y)}g^p(y))^{1/p}\Bigr),
\end{equation}
for every $x,y\in X'$.

Toward this, let $Q_x$ and $Q_y$ be the cubes that contain $x \in \mathbf X'$ and $y \in \mathbf X'$, respectively. Again, we consider the most involved case as the other cases follow similarly. Namely, suppose that we have either case (b) or (c) of Lemma~\ref{QM-cube-case}. Let $\Gamma_{x,y}$ be the path family constructed by concatenating the list of path families in (1)-(9) above.

Each individual $\gamma \in \Gamma_{x,y}$ is a concatenation of subpaths that belong to $\Delta_x, \Gamma(I_x)$, etc. Let $\Gamma$ be any of the latter families. By the expression
$$
\int_\Gamma \int_\gamma \, dsd\gamma
$$
we shall mean that the inside integral occurs over the subpath of $\gamma$ that belongs to $\Gamma$, and the exterior integral is with respect to the probability measure that $\Gamma$ is equipped with, i.e.\ the normalized Lebesgue measure on the interval over which $\Gamma$ is indexed (Remark~\ref{rem:free-par}).

By this notation and the additivity of integral, for every nonnegative measurable $g$ we have
\begin{align*}
    \intavg_{I_x} \int_{\gamma_t} g\, ds dt = \intavg_{\Delta_x} \int_\gamma g\, ds d\gamma & + \intavg_{\Gamma(I_x)} \int_\gamma g\, ds d\gamma +  \intavg_{\Gamma(Q_x;Q_M)} \int_\gamma g\, ds d\gamma \\
    & + \intavg_{\Gamma(I_{m_x-1})}\int_\gamma g\, ds d\gamma +\intavg_{\Delta_M}\int_\gamma g\, ds d\gamma \\
    & +\intavg_{\Gamma(I_{m_y-1})} \int_\gamma g\, ds d\gamma + \intavg_{\Gamma(Q_y;Q_M)} \int_\gamma g\, ds d\gamma \\
     & + \intavg_{\Gamma(I_y)} \int_\gamma g\, ds d\gamma + \intavg_{\Delta_y} \int_\gamma g\, ds d\gamma.
\end{align*}
Now we apply the integral estimates for the individual families that were proved throughout the preceding sections. First, from inequalities \eqref{eq:Delta-x} and \eqref{eq:Delta-y} and H\"older's inequality we have
$$
 \intavg_{\Delta_x} \int_\gamma g\, ds d\gamma \lesssim d(x,y) \Bigl(M_{Cd(x,y)}g^p(x)\Bigr)^{1/p}
$$
and
$$
 \intavg_{\Delta_y} \int_\gamma g\, ds d\gamma \lesssim d(x,y) \Bigl(M_{Cd(x,y)}g^p(y)\Bigr)^{1/p}.
$$
By Proposition~\ref{lem:dbl-intg-cone-full} and the distance estimates of Lemma~\ref{pnt-2},
$$
 \intavg_{\Gamma(Q_x;Q_M)} \int_\gamma g\, ds d\gamma
 \lesssim  d(x,y) \Bigl(M_{Cd(x,y)}g^p(x)\Bigr)^{1/p}
$$
and
$$
\intavg_{\Gamma(Q_y;Q_M)} \int_\gamma g\, ds d\gamma
 \lesssim d(x,y)  \Bigl(M_{Cd(x,y)}g^p(y)\Bigr)^{1/p}.
$$
Lemma~\ref{lem:Delta-M-fam} and H\"older's inequality give
$$
\intavg_{\Delta_M}\int_\gamma g\, ds d\gamma \lesssim d(x,y)  \Bigl((M_{Cd(x,y)}g^p(x))^{1/p}+(M_{Cd(x,y)}g^p(y))^{1/p}\Bigr).
$$
Finally, Proposition~\ref{kern-gamI-thm} and the distance estimates of Lemma~\ref{pnt-2}, yield the same upper bound on each of the remaining four integrals.

Combining the inequalities gives
\begin{equation}\label{eq-a-monday1}
    \intavg_{\Gamma_{x,y}} \int_\gamma g\, ds d\gamma \lesssim d(x,y)\Bigl(\mathcal M_{Cd(x,y)}g^p(x))^{1/p} + \mathcal M_{Cd(x,y)}g^p(y))^{1/p}\Bigr).
\end{equation}
Now if $u\colon \mathbf X \to \mathbb R$ is continuous and $g$ is an upper-gradient of $u$, then
$$
|u(x)-u(y)| \le \int_\gamma g\, ds, \quad \text{for every $\gamma \in \Gamma_{x,y}$.}
$$
Taking integral average on both sides of the latter inequality, over the paths $\gamma \in \Gamma_{x,y}$, and using \eqref{eq-a-monday1} yields
\begin{align*}
|u(x)-u(y)| & \lesssim \intavg_{\Gamma_{x,y}} \int_\gamma g\, ds d\gamma \\
& \lesssim d(x,y)\Bigl(\mathcal M_{Cd(x,y)}g^p(x))^{1/p} + \mathcal M_{Cd(x,y)}g^p(y))^{1/p}\Bigr).
\end{align*}
Since, $C$ and the comparison constants are independent of $x$ and $y$, we have confirmed \eqref{pntws}. We have proved that if $p>\textswab{p}_0$, then $\mathbf X$ supports a $p$-Poincar\'e inequality.

\subsection{Necessity of \protect\boldmath {$p>\textswab{p}_0$}}
\label{secpino}
\begin{proposition}\label{necmain}
  The space $\mathbf X$ corresponding to the choice of $\textswab{C}=\textswab{C}(\lambda)$ and $\textswab{D}=\frac{1}{1-2\lambda^\nu}\textswab{C}(\lambda^\nu),$ where $\lambda \in (0,1/2)$ and $ \nu \in \{2,3,\ldots\}$, does not support a $q$-Poincar\'e inequality for any $q \in [1,\textswab{p}_0]$.
\end{proposition}
\begin{proof}
First, suppose that $q<\textswab{p}_0$. Denote by $J_0$ the largest interval in $\mathcal J$, which will be of unit length with $\inf J_0>0$. Fix $k \in \mathbb N$ and set
$$
\mathcal{J}_k=\left\{ J \in \mathcal{J}: |J| = \lambda^{k\nu}, J \subset 
\Bigl[{\inf J_0}/{2},  \inf J_0\right] \Bigr\}.
$$
Let
$N_k:= \Card \mathcal{J}_k$. From the construction, $N_k \approx 2^{k}$ for large $k$. Let $J^1,\ldots, J^{N_k}$ be the intervals in $\mathcal{J}_k$ ordered such that $\sup J^i < \inf J^{i+1}$ for $i=1,2,\ldots,(N_k-1)$. There exists a monotone function $\psi_k\colon \mbbr \to \mbbr$ such that
    \begin{itemize}
        \item $\psi_k$ is piecewise-linear on $\bigcup_{i=1}^{N_k} J^i$, 
        \item $\psi_k'(y) = \frac{1}{N_k|J^i|}=\frac{1}{{N_k}\lambda^{k\nu}}$ for every $i=1,2,\ldots,(N_k-1)$, every $y\in J^i$,
        \item $\psi_k'\equiv 0$ on the complement of $\bigcup_{i=1}^{N_k}J^i$,
        \item $\psi_k(y) \equiv  0$ for all $y \leq \inf J^1$,
        \item $\psi_k(y) \equiv 1$ for all $ y \geq \sup J^{
        N_k
        }$.
        \end{itemize}      
We define $u_k\colon \mathbf{X} \to \mbbr$ by $u_k(x,y) = \psi_k (y)$. We have,
$$
    \mu (\left\{x\in \mathbf X:u_k(x) = 1 \right\}) \approx \mu (\left\{x\in \mathbf X:u_k(x) = 0 \right\}) \approx \mu(\mathbf{X}).
$$  
We fix a ball $B$ such that $B=\mathbf X$ and $\diam B \approx \diam \mathbf X$. Clearly,
    \begin{align*}
    \int_{B}|u_k-({u_k})_B|d\mu \gtrsim \min\Bigl\{\int_{\{u_k=1 \}}|u_k-{(u_k)}_\mathbf{X}|d\mu, \int_{\{u_k=0 \}}|u_k-{(u_k)}_\mathbf{X}|d\mu\Bigr\}. 
    \end{align*}
For every $k$, $(u_k)_B$ is either less that $1/2$ or more than it. Thus, from the estimates above,
\begin{equation}
\label{fayl1}
    \intavg_{B}|u_k-({u_k})_B|d\mu \gtrsim 1,
\end{equation}
where the constant is independent of $k$. 

On the other hand, one directly checks that for every $k$,
$$
g_k(x,y)= \sum_{i=1}^{N_k}\frac{1}{N_k\lambda^{k\nu}}\chi_{J^i}(y)
$$
is an upper-gradient of $u_k$.
Recall that on row $ J^i $, there are exactly $ 2^{k\nu}$ many cubes of side length $\lambda^{k\nu}$. From these considerations and $N_k\approx 2^k$, we obtain
\begin{align}
    \int_\mathbf{X} g^q d\mu =& \sum_{i=1}^{N_k} \Bigl(\frac{1}{N_k\lambda^{k\nu}}\Bigr)^q \Bigl(2^{k\nu} (\lambda^{k\nu})^2\Bigr) \notag \\
    =&N_k^{1-q}2^{k\nu}\lambda^{-k\nu q+2k\nu} \lesssim (2^k)^{(1+\nu+2\nu\log_2\lambda)-q(1+\nu\log_2\lambda)}. \label{fayl2}
    \end{align}
Condition $q < \textswab{p}_0$ is exactly the sufficient condition for the last expression to converge to zero when $k \to \infty$. Thus, for $q<\textswab{p}_0$, the combination of inequalities \eqref{fayl2} and \eqref{fayl1} rule out the possibility of a $q$-Poincar\'e inequality on $\mathbf X$. So, the proposition is proved in the case of $q<\textswab{p}_0$.

From Keith and Zhong's celebrated result~\cite{keith-zhong}, if a complete metric space, equipped with a doubling measure, supports a $p$-Poincar\'e inequality with $1<p<\infty$, then there exists $\varepsilon>0$ so that the space supports a $(p-\varepsilon)$-Poincar\'e inequality. Thus, from the previous case, we deduce that $\mathbf X$ does not support a $\textswab{p}_0$-Poincar\'e inequality either. Proof of Proposition~\ref{necmain} is complete.
\end{proof}

\section{A Quasiconformal Map on \protect\boldmath {$\mathbf X$}}
\label{sec:qc}
As before, the space $\mathbf X$ is built with respect to the Cantor sets $\textswab{C}=\textswab{C}(\lambda)$ and $\textswab{D}=\frac{1}{1-2\lambda^\nu}\textswab{C}(\lambda^\nu)$.

Let $\phi$ denote the Cantor-Vitali function associated to $\textswab{D}$. It is well-known that $\phi$ is not absolutely continuous; indeed, its restriction to any surviving interval at any step of the the construction of $\textswab{D}$ is not absolutely continuous. Define $f\colon \mathbf X \to \mathbb R^2$ by
\begin{equation}\label{eq:QC-map-f}
       f(x,y):=(x+\phi(y),y).
\end{equation}
Notice that $f$ is an isometry on each cube $I\times J$.

Set $Y:=f(\mathbf X)$ and $E:=\textswab{C}\times \textswab{D}$. We have $\dim E = {\dim \textswab{C} + \dim \textswab{D}}$ and $0< \H^{\dim E}(E) < \infty$ (see \cite{Mattila}). Moreover, the map $f$ is a homeomorphism, and $Y$ is compact and Ahlfors $2$-regular off $f(E)$.
\begin{proposition}
\label{fnotSob}
Let $q>\textswab{p}_0$. Then there exists a family $\Gamma$ of rectifiable paths in $\mathbf X$ such that $\Gamma$ has positive $q$-modulus and that the map $f$ fails to be absolutely continuous on every path in $\Gamma$. In particular, $f \notin N^{1,q}(\mathbf{X};Y)$.
\end{proposition}
\begin{proof}
Fix $\Gamma = \Gamma(Q_1,Q_2)$, the path family associated with some pair of cubes in $\mathbf X$ where $Q_2$ is the first large cube below/above $Q_1$. By Corollary \ref{pos-mod},  $\Gamma(Q_1,Q_2)$ has positive $q$-modulus. Since $\phi(y)$ is the Cantor function, we can show that $f$ is not absolutely continuous on any $\eta_t \in \Gamma(Q_1,Q_2)$. This is basically because absolute continuity of $\phi$ fails exactly along $\textswab{D}$ and each $\eta_t$ crosses the cubes along $\textswab{D}$, therefore, (the first component of) $f \circ \eta_t$ fails to be absolutely continuous (compare to the proof of~\cite[Lemma~5.1]{Kosk-wild}).
\end{proof}

\begin{proof}[Proof of Theorem~\ref{main2}]
Let $p \in (1,2)$ and $\varepsilon>0$ be given. It is enough to prove the claim for small $\eps$. We choose $\lambda \in (0,1/2)$ and $\nu \in \{2,3,\ldots\}$ such that
\begin{equation}
2-\dim \textswab{C} < \textswab{p}_0 < 2-\dim \textswab{C} + \eps.
\end{equation}
and
$$
\dim \textswab{C} + \dim \textswab{D} = 2-p.
$$
Such choices are possible due to the asymptotic behavior of $\textswab{p}_0$ (Remark~\ref{thresh}).

With choices made above for $\textswab{C}$, \textswab{D}, and $E$, we have $\dim E = 2-p$ and $\textswab{p}_0<p+\eps$.

Let $f\colon \mathbf X \to Y$ be the homeomorphism defined in~\eqref{eq:QC-map-f}. Then $H_f(x) = 1$ at every $x \notin E$, hence, a.e.\ on $\mathbf X$. The topological and Ahlfors regularity assumptions hold on $\mathbf X$ and $Y$.

Moreover, $\mathbf X$ supports a
$(p+\varepsilon)$-Poincar\'e inequality by Theorem \ref{main1}, and the exceptional set $E$ satisfies $0<\H^{2-p}(E) < \infty$. However, by Proposition~\ref{fnotSob}, $f$ fails to be absolutely continuous on $q$-a.e.\ path, and hence $f \notin N^{1,q}(\mathbf{X};Y)$, for any $q\ge p+\varepsilon$. Proof of Theorem~\ref{main2} is complete.
\end{proof}

\bibliographystyle{alpha}
\bibliography{Esmayli_Bibliography}

@book {Mattila,
    AUTHOR = {Mattila, Pertti},
     TITLE = {Geometry of sets and measures in {E}uclidean spaces},
    SERIES = {Cambridge Studies in Advanced Mathematics},
    VOLUME = {44},
      NOTE = {Fractals and rectifiability},
 PUBLISHER = {Cambridge University Press, Cambridge},
      YEAR = {1995},
     PAGES = {xii+343},
      ISBN = {0-521-46576-1; 0-521-65595-1},
   MRCLASS = {28A75 (49Q20)},
  MRNUMBER = {1333890},
MRREVIEWER = {Harold\ Parks},
       DOI = {10.1017/CBO9780511623813},
       URL = {https://doi.org/10.1017/CBO9780511623813},
}

@book {Evans-Gariepy,
  AUTHOR = {Evans, Lawrence C. and Gariepy, Ronald F.},
     TITLE = {Measure theory and fine properties of functions},
    SERIES = {Textbooks in Mathematics},
   EDITION = {Revised},
 PUBLISHER = {CRC Press, Boca Raton, FL},
      YEAR = {2015},
     PAGES = {xiv+299},
      ISBN = {978-1-4822-4238-6},
   MRCLASS = {28-01},
  MRNUMBER = {3409135},
}

@article {Shanmu:00,
    AUTHOR = {Shanmugalingam, Nageswari},
     TITLE = {Newtonian spaces: an extension of {S}obolev spaces to metric
              measure spaces},
   JOURNAL = {Rev. Mat. Iberoamericana},
  FJOURNAL = {Revista Matem\'atica Iberoamericana},
    VOLUME = {16},
      YEAR = {2000},
    NUMBER = {2},
     PAGES = {243--279},
      ISSN = {0213-2230},
   MRCLASS = {46E35},
  MRNUMBER = {1809341},
MRREVIEWER = {Daniele\ Morbidelli},
       DOI = {10.4171/RMI/275},
       URL = {https://doi.org/10.4171/RMI/275},
}

@article {graczyk-smirnov,
    AUTHOR = {Graczyk, Jacek and Smirnov, Stanislav},
     TITLE = {Non-uniform hyperbolicity in complex dynamics},
   JOURNAL = {Invent. Math.},
  FJOURNAL = {Inventiones Mathematicae},
    VOLUME = {175},
      YEAR = {2009},
    NUMBER = {2},
     PAGES = {335--415},
      ISSN = {0020-9910,1432-1297},
   MRCLASS = {37F10 (37D25 37F35 37F45)},
  MRNUMBER = {2470110},
MRREVIEWER = {Henk\ Bruin},
       DOI = {10.1007/s00222-008-0152-8},
       URL = {https://doi.org/10.1007/s00222-008-0152-8},
}

@article {przytycki-rohde,
    AUTHOR = {Przytycki, Feliks and Rohde, Steffen},
     TITLE = {Rigidity of holomorphic {C}ollet-{E}ckmann repellers},
   JOURNAL = {Ark. Mat.},
  FJOURNAL = {Arkiv f\"or Matematik},
    VOLUME = {37},
      YEAR = {1999},
    NUMBER = {2},
     PAGES = {357--371},
      ISSN = {0004-2080,1871-2487},
   MRCLASS = {37F10 (37C15 37F30)},
  MRNUMBER = {1714763},
MRREVIEWER = {Peter\ Ha\"issinsky},
       DOI = {10.1007/BF02412220},
       URL = {https://doi.org/10.1007/BF02412220},
}

@article {kozlovski-shen-strien,
    AUTHOR = {Kozlovski, O. and Shen, W. and van Strien, S.},
     TITLE = {Rigidity for real polynomials},
   JOURNAL = {Ann. of Math. (2)},
  FJOURNAL = {Annals of Mathematics. Second Series},
    VOLUME = {165},
      YEAR = {2007},
    NUMBER = {3},
     PAGES = {749--841},
      ISSN = {0003-486X,1939-8980},
   MRCLASS = {37E05 (30C10 30C62 37C20 37E20 37F10 37F30 37F50)},
  MRNUMBER = {2335796},
MRREVIEWER = {Henk\ Bruin},
       DOI = {10.4007/annals.2007.165.749},
       URL = {https://doi.org/10.4007/annals.2007.165.749},
}

@article {haissinsky,
    AUTHOR = {Haissinsky, Peter},
     TITLE = {Rigidity and expansion for rational maps},
   JOURNAL = {J. London Math. Soc. (2)},
  FJOURNAL = {Journal of the London Mathematical Society. Second Series},
    VOLUME = {63},
      YEAR = {2001},
    NUMBER = {1},
     PAGES = {128--140},
      ISSN = {0024-6107,1469-7750},
   MRCLASS = {37F15 (37F30)},
  MRNUMBER = {1802762},
MRREVIEWER = {Lei\ Tan},
       DOI = {10.1112/S0024610700001563},
       URL = {https://doi.org/10.1112/S0024610700001563},
}

@article {smania,
    AUTHOR = {Smania, Daniel},
     TITLE = {Puzzle geometry and rigidity: the {F}ibonacci cycle is
              hyperbolic},
   JOURNAL = {J. Amer. Math. Soc.},
  FJOURNAL = {Journal of the American Mathematical Society},
    VOLUME = {20},
      YEAR = {2007},
    NUMBER = {3},
     PAGES = {629--673},
      ISSN = {0894-0347,1088-6834},
   MRCLASS = {37E20 (30C62 30C65 37C15 37F25 37F45)},
  MRNUMBER = {2291915},
MRREVIEWER = {Henk\ Bruin},
       DOI = {10.1090/S0894-0347-07-00550-4},
       URL = {https://doi.org/10.1090/S0894-0347-07-00550-4},
}

@incollection {kosk-wild-survey,
    AUTHOR = {Koskela, P. and Wildrick, K.},
     TITLE = {Analytic properties of quasiconformal mappings between metric
              spaces},
 BOOKTITLE = {Metric and differential geometry},
    SERIES = {Progr. Math.},
    VOLUME = {297},
     PAGES = {163--174},
 PUBLISHER = {Birkh\"auser/Springer, Basel},
      YEAR = {2012},
      ISBN = {978-3-0348-0256-7; 978-3-0348-0257-4},
   MRCLASS = {30L10 (30C65 46E35)},
  MRNUMBER = {3220442},
       DOI = {10.1007/978-3-0348-0257-4\_6},
       URL = {https://doi.org/10.1007/978-3-0348-0257-4_6},
}

@article {Williams:14,
    AUTHOR = {Williams, Marshall},
     TITLE = {Dilatation, pointwise {L}ipschitz constants, and condition
              {$N$} on curves},
   JOURNAL = {Michigan Math. J.},
  FJOURNAL = {Michigan Mathematical Journal},
    VOLUME = {63},
      YEAR = {2014},
    NUMBER = {4},
     PAGES = {687--700},
      ISSN = {0026-2285,1945-2365},
   MRCLASS = {26B30 (30C65)},
  MRNUMBER = {3286666},
MRREVIEWER = {Daniel\ Meyer},
       DOI = {10.1307/mmj/1417799221},
       URL = {https://doi.org/10.1307/mmj/1417799221},
}

@article {lahti:Zhou:24,
    AUTHOR = {Lahti, Panu and Zhou, Xiaodan},
     TITLE = {Metric quasiconformality and {S}obolev regularity in
              non-{A}hlfors regular spaces},
   JOURNAL = {Anal. Geom. Metr. Spaces},
  FJOURNAL = {Analysis and Geometry in Metric Spaces},
    VOLUME = {12},
      YEAR = {2024},
    NUMBER = {1},
     PAGES = {Paper No. 20240001, 22},
      ISSN = {2299-3274},
   MRCLASS = {30L10 (30C65 46E36)},
  MRNUMBER = {4733790},
MRREVIEWER = {David\ Matthew\ Freeman},
       DOI = {10.1515/agms-2024-0001},
       URL = {https://doi.org/10.1515/agms-2024-0001},
}

@article {Ntala24:metric-def,
    AUTHOR = {Ntalampekos, Dimitrios},
     TITLE = {Metric definition of quasiconformality and exceptional sets},
   JOURNAL = {Math. Ann.},
  FJOURNAL = {Mathematische Annalen},
    VOLUME = {389},
      YEAR = {2024},
    NUMBER = {3},
     PAGES = {3231--3253},
      ISSN = {0025-5831,1432-1807},
   MRCLASS = {30C62 (30C65 31A15 31B15)},
  MRNUMBER = {4753085},
       DOI = {10.1007/s00208-023-02723-6},
       URL = {https://doi.org/10.1007/s00208-023-02723-6},
}

@article {Kallunki-Kosk-2003,
    AUTHOR = {Kallunki, S. and Koskela, P.},
     TITLE = {Metric definition of {$\mu$}-homeomorphisms},
   JOURNAL = {Michigan Math. J.},
  FJOURNAL = {Michigan Mathematical Journal},
    VOLUME = {51},
      YEAR = {2003},
    NUMBER = {1},
     PAGES = {141--151},
      ISSN = {0026-2285,1945-2365},
   MRCLASS = {30C62 (30C65)},
  MRNUMBER = {1960925},
MRREVIEWER = {M.\ Yu.\ Vasil\cprime chik},
       DOI = {10.1307/mmj/1049832897},
       URL = {https://doi.org/10.1307/mmj/1049832897},
}

@article {Kallunki-Kosk-2000,
    AUTHOR = {Kallunki, Sari and Koskela, Pekka},
     TITLE = {Exceptional sets for the definition of quasiconformality},
   JOURNAL = {Amer. J. Math.},
  FJOURNAL = {American Journal of Mathematics},
    VOLUME = {122},
      YEAR = {2000},
    NUMBER = {4},
     PAGES = {735--743},
      ISSN = {0002-9327,1080-6377},
   MRCLASS = {37F30 (30C65)},
  MRNUMBER = {1771571},
MRREVIEWER = {Peter\ Ha\"issinsky},
       URL =
              {http://muse.jhu.edu/journals/american_journal_of_mathematics/v122/122.4kallunki.pdf},
}

@article {Hei-Kosk-95,
    AUTHOR = {Heinonen, Juha and Koskela, Pekka},
     TITLE = {Definitions of quasiconformality},
   JOURNAL = {Invent. Math.},
  FJOURNAL = {Inventiones Mathematicae},
    VOLUME = {120},
      YEAR = {1995},
    NUMBER = {1},
     PAGES = {61--79},
      ISSN = {0020-9910,1432-1297},
   MRCLASS = {30C65 (30C70)},
  MRNUMBER = {1323982},
MRREVIEWER = {J.\ Ferrand},
       DOI = {10.1007/BF01241122},
       URL = {https://doi.org/10.1007/BF01241122},
}

@article {Haj:Ko:met,
    AUTHOR = {Haj\l{}asz, Piotr and Koskela, Pekka},
     TITLE = {Sobolev met {P}oincar\'e},
   JOURNAL = {Mem. Amer. Math. Soc.},
  FJOURNAL = {Memoirs of the American Mathematical Society},
    VOLUME = {145},
      YEAR = {2000},
    NUMBER = {688},
     PAGES = {x+101},
      ISSN = {0065-9266,1947-6221},
   MRCLASS = {46E35 (30C65 31C25 53C17 58J60)},
  MRNUMBER = {1683160},
MRREVIEWER = {Alexander\ D.\ Ukhlov},
       DOI = {10.1090/memo/0688},
       URL = {https://doi.org/10.1090/memo/0688},
}

@book {HKST:15,
    AUTHOR = {Heinonen, Juha and Koskela, Pekka and Shanmugalingam,
              Nageswari and Tyson, Jeremy T.},
     TITLE = {Sobolev spaces on metric measure spaces},
    SERIES = {New Mathematical Monographs},
    VOLUME = {27},
      NOTE = {An approach based on upper gradients},
 PUBLISHER = {Cambridge University Press, Cambridge},
      YEAR = {2015},
     PAGES = {xii+434},
      ISBN = {978-1-107-09234-1},
   MRCLASS = {30-02 (30L05 30L10 31E05 46E35)},
  MRNUMBER = {3363168},
MRREVIEWER = {David Matthew Freeman},
       DOI = {10.1017/CBO9781316135914},
       URL = {https://doi.org/10.1017/CBO9781316135914},
}

@book {Hei:01,
    AUTHOR = {Heinonen, Juha},
     TITLE = {Lectures on analysis on metric spaces},
    SERIES = {Universitext},
 PUBLISHER = {Springer-Verlag, New York},
      YEAR = {2001},
     PAGES = {x+140},
      ISBN = {0-387-95104-0},
   MRCLASS = {30C65 (28A75 28A78 46E35)},
  MRNUMBER = {1800917},
MRREVIEWER = {Christopher Bishop},
       DOI = {10.1007/978-1-4613-0131-8},
       URL = {https://doi.org/10.1007/978-1-4613-0131-8},
}

@article {Geh:63,
    AUTHOR = {Gehring, F. W.},
     TITLE = {Rings and quasiconformal mappings in space},
   JOURNAL = {Trans. Amer. Math. Soc.},
  FJOURNAL = {Transactions of the American Mathematical Society},
    VOLUME = {103},
      YEAR = {1962},
     PAGES = {353--393},
      ISSN = {0002-9947,1088-6850},
   MRCLASS = {30.47},
  MRNUMBER = {139735},
MRREVIEWER = {L.\ V.\ Ahlfors},
       DOI = {10.2307/1993834},
       URL = {https://doi.org/10.2307/1993834},
}

@article {Geh1,
    AUTHOR = {Gehring, F. W.},
     TITLE = {The definitions and exceptional sets for quasiconformal
              mappings},
   JOURNAL = {Ann. Acad. Sci. Fenn. Ser. A I},
  FJOURNAL = {Ann. Acad. Sci. Fenn. Ser. A I},
    VOLUME = {281},
      YEAR = {1960},
     PAGES = {28},
   MRCLASS = {30.47},
  MRNUMBER = {124488},
MRREVIEWER = {A.\ Pfluger},
}

@article {HeiKo-Acta,
    AUTHOR = {Heinonen, Juha and Koskela, Pekka},
     TITLE = {Quasiconformal maps in metric spaces with controlled geometry},
   JOURNAL = {Acta Math.},
  FJOURNAL = {Acta Mathematica},
    VOLUME = {181},
      YEAR = {1998},
    NUMBER = {1},
     PAGES = {1--61},
      ISSN = {0001-5962,1871-2509},
   MRCLASS = {30C65 (46E99)},
  MRNUMBER = {1654771},
MRREVIEWER = {M.\ Yu.\ Vasil\cprime chik},
       DOI = {10.1007/BF02392747},
       URL = {https://doi.org/10.1007/BF02392747},
}

@article {keith-mod,
    AUTHOR = {Keith, Stephen},
     TITLE = {Modulus and the {P}oincar\'{e} inequality on metric measure
              spaces},
   JOURNAL = {Math. Z.},
  FJOURNAL = {Mathematische Zeitschrift},
    VOLUME = {245},
      YEAR = {2003},
    NUMBER = {2},
     PAGES = {255--292},
      ISSN = {0025-5874,1432-1823},
   MRCLASS = {31C15 (46E35)},
  MRNUMBER = {2013501},
MRREVIEWER = {Jana\ Bj\"{o}rn},
       DOI = {10.1007/s00209-003-0542-y},
       URL = {https://doi.org/10.1007/s00209-003-0542-y},
}

@article {keith-zhong,
    AUTHOR = {Keith, Stephen and Zhong, Xiao},
     TITLE = {The {P}oincar\'{e} inequality is an open ended condition},
   JOURNAL = {Ann. of Math. (2)},
  FJOURNAL = {Annals of Mathematics. Second Series},
    VOLUME = {167},
      YEAR = {2008},
    NUMBER = {2},
     PAGES = {575--599},
      ISSN = {0003-486X,1939-8980},
   MRCLASS = {46E35 (30C65 43A85)},
  MRNUMBER = {2415381},
MRREVIEWER = {Jeremy\ T.\ Tyson},
       DOI = {10.4007/annals.2008.167.575},
       URL = {https://doi.org/10.4007/annals.2008.167.575},
}

@article {MacTysWil,
    AUTHOR = {Mackay, John M. and Tyson, Jeremy T. and Wildrick, Kevin},
     TITLE = {Modulus and {P}oincar\'e{} inequalities on non-self-similar
              {S}ierpi\'nski carpets},
   JOURNAL = {Geom. Funct. Anal.},
  FJOURNAL = {Geometric and Functional Analysis},
    VOLUME = {23},
      YEAR = {2013},
    NUMBER = {3},
     PAGES = {985--1034},
      ISSN = {1016-443X,1420-8970},
   MRCLASS = {30L10 (28A80 31E05)},
  MRNUMBER = {3061778},
MRREVIEWER = {Matthew\ Badger},
       DOI = {10.1007/s00039-013-0227-6},
       URL = {https://doi.org/10.1007/s00039-013-0227-6},
}

@article {Kosk-wild,
    AUTHOR = {Koskela, P. and Wildrick, K.},
     TITLE = {Exceptional sets for quasiconformal mappings in general metric
              spaces},
   JOURNAL = {Int. Math. Res. Not. IMRN},
  FJOURNAL = {International Mathematics Research Notices. IMRN},
      YEAR = {2008},
    NUMBER = {9},
     PAGES = {Art. ID rnn020, 32},
      ISSN = {1073-7928,1687-0247},
   MRCLASS = {30C65 (26B30 54E40)},
  MRNUMBER = {2429239},
MRREVIEWER = {Leonid\ V.\ Kovalev},
       DOI = {10.1093/imrn/rnn020},
       URL = {https://doi.org/10.1093/imrn/rnn020},
}

@article {bal-kosk-rog,
    AUTHOR = {Balogh, Zolt\'{a}n M. and Koskela, Pekka and Rogovin, Sari},
     TITLE = {Absolute continuity of quasiconformal mappings on curves},
   JOURNAL = {Geom. Funct. Anal.},
  FJOURNAL = {Geometric and Functional Analysis},
    VOLUME = {17},
      YEAR = {2007},
    NUMBER = {3},
     PAGES = {645--664},
      ISSN = {1016-443X,1420-8970},
   MRCLASS = {30C65 (37F10)},
  MRNUMBER = {2346270},
MRREVIEWER = {Volker\ Mayer},
       DOI = {10.1007/s00039-007-0607-x},
       URL = {https://doi.org/10.1007/s00039-007-0607-x},
}
\end{document}